\documentclass[preprint,review,12pt]{elsarticle}
\usepackage{amssymb}
\usepackage{amsmath}
\usepackage{amsthm}
\usepackage{cases}
\usepackage{graphicx}
\usepackage{subfigure}
\usepackage{epstopdf}
\usepackage{epsfig}

\usepackage{mathrsfs}
\usepackage{mathabx}

\usepackage{color}

\usepackage{diagbox}
\usepackage{multirow}

\newtheorem*{assumption*}{Assumption}

\newtheorem{assumption}{Assumption}

\newtheorem{lemma}{Lemma}
\newtheorem{remark}{Remark}
\newtheorem{theorem}{Theorem}

\begin{document}

\begin{frontmatter}
	
	\title{Supercloseness of finite element method for a singularly perturbed convection-diffusion problem on Bakhvalov-type mesh in 2D \tnoteref{funding}}
	\tnotetext[funding]{The current research  was partly supported by NSFC (11771257), Shandong Provincial NSF (ZR2021MA004).}
	\author[label1] {Chunxiao Zhang\fnref{cor1}}
	\author[label1] {Jin Zhang\corref{cor2}}
	\fntext[cor1] {Email address:~chunxiaozhangang@outlook.com }
	\cortext[cor2] {Email address:~jinzhangalex@sdnu.edu.cn }
	\address[label1]{School of Mathematics and Statistics, Shandong Normal University,
		Jinan 250014, China}
	


	\begin{abstract}
For singularly perturbed convection-diffusion problems, supercloseness analysis of finite element method is still open on Bakhvalov-type meshes, especially in the case of 2D. The difficulties arise from the width of the mesh in the layer adjacent to the transition point, resulting in a suboptimal estimate for convergence. Existing analysis techniques cannot handle these difficulties well. To fill this gap, a novel interpolation is designed delicately for the first time for the smooth part of the solution, bringing about the optimal supercloseness result of almost order 2 under an energy norm for finite element method. Our theoretical result is uniformly in the singular perturbation parameter $\varepsilon$ and is supported by the numerical experiments.

	\end{abstract}                                                                                                                                          
	
	
\end{frontmatter}


\section{Introduction}
The following singularly perturbed convection-diffusion equation is taken into consideration:
\begin{equation}\label{model problem}
	\begin{aligned}
		-\varepsilon\Delta u-bu_{x}+cu=&f\quad&&\text{in $\Omega=(0,1)^{2}$},\\
		u=&0\quad &&\text{on $\partial\Omega$},
	\end{aligned}
\end{equation}
where $0<\varepsilon\ll1$ represents the perturbation parameter, and $b$, $c$, and $f$ are bivariate functions with sufficient smoothness. The following assumptions also hold:
$$b\geq\beta>0,\quad c+\frac{1}{2}b_{x}\ge\gamma>0\quad \text{on $\overline{\Omega}$},$$
where $\beta$, $\gamma$ are constants. According to the above assumptions, solution to problem \eqref{model problem} exists and is unique in $H_{0}^{1}(\Omega)\cap H^{2}(\Omega)$ for each $f\in L^{2}(\Omega)$ (see \cite{Roo1Sty2:2008-Robust}). Besides, this solution has two parabolic layers at $y=0$, $y=1$ of width $\mathcal{O}(\sqrt\varepsilon\ln\frac{1}{\varepsilon})$ and an exponential layer at $x=0$ of width $\mathcal{O}(\varepsilon\ln\frac{1}{\varepsilon})$. Problems with characteristic layers like \eqref{model problem} are crucial. They can serve as mathematical models for dealing with practical problems, for example, the flow past a surface with a no-slip border condition. It is well known that the existence of boundary layers makes standard numerical methods unstable. A large number of viewpoints are proposed to deal with layers. Among them, applying finite element method to layer-adapted meshes is an essential subject \cite{Roo1:1998-Layer, Roo1Lin2:1999-Sufficient, Lin1:2003-Layer, Lin1Tor2:2009-Analysis, Gra1Lis2:2010-Coupled}. 

For finite element method (FEM), supercloseness is a crucial property of convergence. Supercloseness here means that the convergence property of $\varPi u-u^N$ is higher than that of $u-u^N$ in a certain norm, where $u$ is the exact solution, $\varPi u$ is a certain interpolation of $u$ and $u^N$ is the finite element solution. Supercloseness property can be employed in various estimates, such as $L^2$ estimates \cite{Sty1Mar2:2003-SDFEM} and $L^\infty$ estimates \cite{Liu1Zhang2:2018-Pointwise}, to evaluate the accuracy of the numerical solutions.

Up to now, there have been numerous studies focusing on supercloseness of FEM on Shishkin-type meshes, one of the layer-adapted meshes \cite{Fra1Lin2:2008-Supercloseness, Zhang1Sty2:2017-Supercloseness, Liu1Sty2:2018-Supercloseness, Li1:2001-Convergence}. Bakhvalov-type meshes are another widely used layer-adapted meshes \cite{linss:2009-layer}. Compared to Shishkin-type meshes, Bakhvalov-type meshes perform numerically better in certain cases, and their transition points are not impacted by the mesh parameter $N$. Nevertheless, little progress has been achieved in the supercloseness analysis on these meshes. The main difficulty arises from the width of the last element mesh in the layer domain, resulting in a suboptimal estimate when applying the standard Lagrange interpolation to Bakhvalov-type meshes (refer to \cite[Question 4.1]{Roo1Sty2:2015-Some} and \cite{Zhang1Liu2:2021-Supercloseness} for more information). To optimize the convergence order on Bakhvalov-type meshes, Zhang and Liu proposed a special interpolation in \cite{Zhang1Liu2:2020-Optimal} for the exponential layer of the solution, which opens up a new direction for the convergence analysis. Subsequently, they created another special interpolation for the smooth part of the solution in \cite{zhang1Liu2:2022-Supercloseness} to improve its convergence accuracy. Combined with the two interpolations, supercloseness of optimal order 2 can be obtained directly in the one-dimensional case. 

As we learned, however, almost all supercloseness analysis on Bakhvalov-type meshes are carried out in 1D \cite{zhang1Liu2:2022-Supercloseness, Zhang1Lv2:2021-Supercloseness}. So far, no article has made supercloseness analysis on Bakhvalov-type meshes in 2D. We find that getting a sharp estimate for the smooth part of the solution faces a lot of challenges in a two-dimensional setting. Existing analysis techniques cannot deal with these difficulties well. Specifically, the interpolation proposed in \cite{zhang1Liu2:2022-Supercloseness} for the smooth part cannot be generalized to a two-dimensional case, indicating that constructing an interpolation that is both continuous and capable of optimizing the estimate of the smooth part is a very challenging task. 

To fill this gap, we construct a novel interpolation according to the characteristics of the smooth part and the structures of Bakhvalov-type meshes in 2D for the first time. The new interpolation can improve the supercloseness accuracy of the smooth part by eliminating estimates of some intractable terms, providing a powerful tool for future supercloseness analysis on Bakhvalov-type meshes in 2D. In addition, we extend another special interpolation in \cite{Zhang1Liu2:2020-Optimal} to a two-dimensional case for the exponential layer of the solution. Note that some boundary corrections are necessary to ensure the interpolation in 2D satisfies homogeneous Dirichlet boundary conditions. Then we can draw the main conclusion of this paper: for problem \eqref{model problem}, the convergence order of supercloseness under an energy norm can attain to almost 2 on a Bakhvalov-type mesh in 2D. Our result is optimal and is uniformly in the singular perturbation parameter $\varepsilon$. 

Then we will introduce the structure of our article. In Section \ref{sec. 2} we state some information about the solution $u$, construct a Bakhvalov-type mesh in 2D, and present the corresponding bilinear FEM. In Section \ref{sec. 3}, the novel interpolation is constructed for our supercloseness arguments. In addition, some preliminary conclusions are derived. We thoroughly demonstrate supercloseness property on the Bakhvalov-type mesh in Section \ref{sec. 4}. Finally, some numerical results are shown in Section \ref{sec. 5} to support our theoretical conclusion.

Throughout this article, the standard Sobolev spaces $W^{m,p}(D)$ will be used, where m are nonnegative integers, domain $D\subset\Omega$ and $H^{m}(D)=W^{m,2}(D)$. On $H^{m}(D)$, $\vert\cdot\vert_{m,D}$ is semi-norm and $\Vert\cdot\Vert_{m,D}$ is norm. In addition, $L^{p}(D)$ is Lebesgue space and $\Vert\cdot\Vert_{L^{p}(D)}$ represents norms in this space. When $p=2$, $\Vert\cdot\Vert_{L^{2}(D)}$ is denoted as $\Vert\cdot\Vert_{D}$. We shall remove $D$ from these notations if $D=\Omega$. Additionally, all constants in this paper, including generic constants $C$ and fixed constants $C_{i}$, are unaffected by $\varepsilon$ and the mesh parameter N.
\section{Regularity results, Bakhvalov-type mesh and FEM}\label{sec. 2}
\subsection{Regularity results}
\begin{assumption}\label{bound}
	We can decompose the solution to \eqref{model problem} into the following parts:
	\begin{equation}\label{decomposition of u}
		u=S+E_{1}+E_{2}+E_{12},\quad \forall(x,y)\in\overline{\Omega},
	\end{equation}
	where $S$ is the smooth part, $E_{1}$ is the exponential layer at $x=0$, $E_{2}$ are the parabolic layers at $y=0$ and $y=1$, $E_{12}$ is the corner layer.
	
	Furthermore, for any $(x,y)\in \bar{\Omega}$, the following inequalities hold:
	\begin{subequations}
		\begin{align}
			\left|\frac{\partial^{m+n}S}{\partial x^{m}\partial y^{n}}(x,y) \right| &\le C,\\
			\left|\frac{\partial^{m+n}E_{1}}{\partial x^{m}\partial y^{n}}(x,y) \right| &\le C\varepsilon^{-m}e^{\frac{-\beta x}{\varepsilon}},\\
			\left|\frac{\partial^{m+n}E_{2}}{\partial x^{m}\partial y^{n}}(x,y) \right| &\le C\varepsilon^{-\frac{n}{2}}(e^{-\frac{y}{\sqrt\varepsilon}}+e^{-\frac{1-y}{\sqrt\varepsilon}}),\\
			\left|\frac{\partial^{m+n}E_{12}}{\partial x^{m}\partial y^{n}}(x,y) \right| &\le C\varepsilon^{-(m+\frac{n}{2})}e^{\frac{-\beta x}{\varepsilon}}(e^{-\frac{y}{\sqrt\varepsilon}}+e^{-\frac{1-y}{\sqrt\varepsilon}}),
		\end{align}
	\end{subequations}
	where $m$, $n$ are nonnegative integers and $0\le m+n\le 3$.
\end{assumption}
More details about Assumption \ref{bound} can be found in \cite{Kel1Sty2:2005-Corner} and \cite{Kel1Sty2:2007-Sharpened}.

\subsection{Bakhvalov-type mesh}
Different Bakhvalov-type meshes are proposed as approximations of Bakhvalov meshes (see \cite{linss:2009-layer}). In this paper, we shall analyze a Bakhvalov-type mesh in 2D introduced in \cite{Roo1:2006-Error}.

Assume that $N\ge4$ is a positive integer and is even, $x_{i}$ ($0\le i\le N$) and $y_{j}$ ($0\le j\le N$) are mesh nodes satisfying  $$0=x_{0}<x_{1}<\cdots<x_{N}=1,\quad0=y_{0}<y_{1}<\cdots<y_{N}=1.$$
Then, the Bakhvalov-type mesh is defined by
\begin{subequations}\label{B-type}
\begin{align}
x_{i}=\phi(\frac{i}{N})=\left\{
\begin{aligned}
	&-\frac{\sigma\varepsilon}{\beta}\ln(1-2(1-\varepsilon)\frac{i}{N})&&i=0,1,...,\frac{N}{2},\\
	&1-(1-x_{\frac{N}{2}})\frac{2(N-i)}{N}&&i=\frac{N}{2}+1,...,N,
\end{aligned}
\right.
\end{align}
	\begin{align}
	y_{j}=\varphi(\frac{j}{N})=\left\{
	\begin{aligned}
		&-\sigma\sqrt\varepsilon\ln(1-4(1-\varepsilon)\frac{j}{N})&&j=0,1,...,\frac{N}{4},\\
		&d_{1}(\frac{j}{N}-\frac{1}{4})+d_{2}(\frac{j}{N}-\frac{3}{4})&&j=\frac{N}{4}+1,...,\frac{3N}{4}-1,\\
		&1+\sigma\sqrt\varepsilon\ln(1-4(1-\varepsilon)(1-\frac{j}{N})),&&j=\frac{3N}{4},...,N,
	\end{aligned}
	\right.
\end{align}
\end{subequations}
where $d_{1}$, $d_{2}$ are used to ensure the continuity of the function $\psi$ at the transition points $\frac{1}{4}$ and $\frac{3}{4}$.

$\Omega$ is split into rectangular meshes by connecting the mesh nodes with lines that are parallel to the $x$-axis and $y$-axis. These meshes are written as $K_{i,j}=[x_i,x_{i+1}]\times[y_j,y_{j+1}]$, and the notation $K$ denotes a generic rectangular mesh. Then we get a tensor-product rectangular mesh $\mathcal{T}_{N}$ with mesh points $(x_i,y_j)$. See Figure \ref{fig:rectangulation}

\begin{figure}[htbp]
	\centering
	\includegraphics[width=0.7\linewidth]{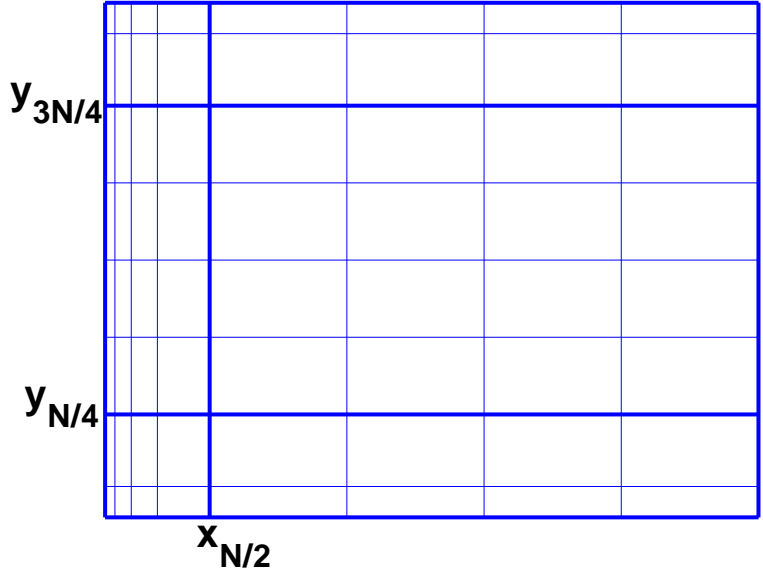}
	\caption{Rectangulation $\mathcal{T}_{N}$}
	\label{fig:rectangulation}
\end{figure}
 
\begin{assumption}\label{transposition}
	Suppose that $\sigma=\frac{5}{2}$ and $\varepsilon\le N^{-1}$ in our analysis. Additionally, we assume that
	$$\frac{\sigma\varepsilon}{\beta}\ln\frac{1}{\varepsilon}\le\frac{1}{2}\quad\text{and}\quad\sigma\sqrt\varepsilon\ln\frac{1}{\varepsilon}\le\frac{1}{4}.$$
\end{assumption}

For convenience, let $h_{x,i}=x_{i+1}-x_{i}$, $h_{y,j}=y_{j+1}-y_{j}$, and when $j=\frac{N}{4},\dots,\frac{3N}{4}-1$, set $h_{y,j}=:h$; when $i=\frac{N}{2},\dots,N-1$, set $h_{x,i}=:H$.
\begin{lemma}\label{mesh step}
	If Assumption \ref{transposition} holds, then for the Bakhvalov-type mesh \eqref{B-type}, we have
	\begin{align*}
		&N^{-1}\le H\le 2N^{-1},\\
		&\frac{1}{2}\sigma\varepsilon\le h_{x,\frac{N}{2}-1}\le 2\sigma N^{-1},\\
		&h_{x,0}\le h_{x,1}\le\cdots\le h_{x,\frac{N}{2}-2}\le \sigma\varepsilon,\\
		&C_1 \varepsilon N^{-1}\le h_{x,0}\le C_2 \varepsilon N^{-1},\\
		&C_3 \sigma\varepsilon\ln N\le x_{\frac{N}{2}-1}\le C_4 \sigma\varepsilon\ln N,\quad x_{\frac{N}{2}}\ge C\sigma\varepsilon|\ln\varepsilon|,\\
		&h_{x,i}^{\mu}e^{\frac{-\beta x_{i}}{\varepsilon}}\le C\varepsilon^{\mu}N^{-\mu},\quad\text{$0\le i\le \frac{N}{2}-2$\quad and\quad$0\le \mu\le \sigma$},\\
		&\vert E_{1}(x_{\frac{N}{2}},y)\vert\le C\varepsilon^{\sigma},\quad\vert E_{1}(x_{\frac{N}{2}-1},y)\vert\le CN^{-\sigma}.
	\end{align*}
	Concerning $h_{y,j}$, $0\le j\le N-1$, except that $\varepsilon$ is replaced by $\sqrt{\varepsilon}$, other properties can be derived similarly.
\end{lemma}

Furthermore, we also need the following property for our subsequent analysis.
\begin{lemma}\label{special step}
	If Assumption \ref{transposition} holds, we can draw the following conclusions:
	\begin{equation}
		\begin{aligned}
			&h_{x,\frac{N}{2}-1}\le C\varepsilon^{1-\alpha}N^{-\alpha},\\
			&h_{y,\frac{N}{4}-1}=h_{y,\frac{3N}{4}}\le C(\sqrt{\varepsilon})^{1-\alpha}N^{-\alpha},
		\end{aligned}
	\end{equation}
	where $\alpha\in \left( 0,1\right] $.
\end{lemma}
\begin{proof}
According to \eqref{B-type}, we can derive that
\begin{equation}
h_{x,\frac{N}{2}-1}=\frac{\sigma\varepsilon}{\beta}\ln\frac{1-2(1-\varepsilon)(\frac{1}{2}-N^{-1})}{1-2(1-\varepsilon)\frac{1}{2}}\le \frac{\sigma\varepsilon}{\beta}\frac{1}{\alpha}(\frac{\varepsilon+2(1-\varepsilon)N^{-1}}{\varepsilon})^\alpha\le C\varepsilon^{1-\alpha}N^{-\alpha},
\end{equation}
here we use the standard arguments $\ln x\le \frac{x^\alpha}{\alpha}\quad \alpha\in (0,1]$. $h_{y,\frac{N}{4}-1}$ and $h_{y,\frac{3N}{4}}$ can be proved in a similar way.
\end{proof}
Then, we divide $\Omega$ into four subdomains:
\begin{align*}
	&\Omega_{s}:=\left[x_{\frac{N}{2}-1},1\right]\times\left[y_{\frac{N}{4}-1},y_{\frac{3N}{4}+1}\right],&&
	\Omega_{y}:=\left[x_{\frac{N}{2}-1},1 \right]\times(\left[0,y_{\frac{N}{4}-1} \right]\cup\left[y_{\frac{3N}{4}+1},1\right] ),\\
	&\Omega_{x}:=\left[0,x_{\frac{N}{2}-1}\right]\times\left[ y_{\frac{N}{4}-1},y_{\frac{3N}{4}+1} \right],&&
	\Omega_{xy}:=\left[0,x_{\frac{N}{2}-1}\right]\times(\left[0,y_{\frac{N}{4}-1} \right]\cup\left[y_{\frac{3N}{4}+1},1\right] ).
\end{align*}
In addition, we define $\Omega_{0}$ as $\left[x_{\frac{N}{2}-1},x_{\frac{N}{2}} \right]\times\left[0,1 \right] $ for convenience.
\subsection{FEM}
Suppose that $V^{N}\subset H_{0}^{1}(\Omega)$ is the piecewise bilinear finite element space of \eqref{model problem}, whose definition is
$$V^{N}=\left\lbrace v\in C(\overline{\Omega } ):\ v|_{\partial\Omega=0}\ \text{and}\ v|_{K}\in \mathcal{Q}_{1}(K),\ \forall K\in \mathcal{T}_{N} \right\rbrace,$$ where
\begin{equation*}
\mathcal{Q}_1(K)=\text{span}\left\lbrace x^iy^j:0\le i,j\le1\right\rbrace.
\end{equation*}
Then the corresponding FEM is: Find $u^N\in V^N$ satisfying $a(u^N,v)=(f,v)\quad\forall v\in V^N$, where
\begin{equation}\label{define u}
	a(u^N,v):=\varepsilon(\nabla u^N,\nabla v)+(-bu_{x}^N,v)+(cu^N,v)\quad\forall v\in V^N.
\end{equation}
According to the conditions of \eqref{model problem}, the following coercivity holds 
\begin{equation}\label{coercivity}
	a(v,v)\ge C\Vert v \Vert^2_{\varepsilon},\qquad \forall v\in V^{N},
\end{equation}
where $C$ is independent of $\varepsilon$ and $N$, and the energy norm is defined as
$$\Vert v\Vert_{\varepsilon}^{2}:=\varepsilon\left| v\right|_{1}^{2}+\Vert v\Vert^{2}.$$
It follows that $u^N$ is well defined by \eqref{define u} (see \cite{Bre1Sus2:2008-Mathematical} and references therein).

In this paper, we will analyze supercloseness property under the energy norm. Let $v=u^{N}-\varPi u$, where $u^N$ is the finite element solution and $\varPi u$ is the interpolation of $u$. From \eqref{coercivity} and the Galerkin orthogonality property $a(u^{N}-u,v)=0,\, \forall v\in V^{N}$ one has
\begin{equation}\label{variation}
	C\Vert v\Vert_{\varepsilon}^{2}\le a(v,v)=a(u-\varPi u,v)= \varepsilon(\nabla(u-\varPi u),\nabla v)+(-b(u-\varPi u)_{x},v)+(c(u-\varPi u),v).
\end{equation}

\section{Interpolation errors and new interpolation}\label{sec. 3}
In this section, we will give some standard interpolation errors and construct a new interpolation for the smooth part of the solution. Additionally, some preliminary conclusions will be presented.

\subsection{Interpolation errors}
The following lemma can refer to \cite{Guo1Sty2:1997-Pointwise} for more information.
\begin{lemma}\label{interpolation error}
Let $K\in \mathcal{T}_{N}$ and suppose that $K$ is $K_{i,j}$. Assume that $w\in W^{2,p}(\Omega)$ and denote by $w^{I}$ the standard Lagrange interpolation of $w$ at the vertices of $K$. Then, 
	\begin{equation*}
		\begin{aligned}
			&\Vert
			w-w^{I}\Vert_K\le C\sum_{s+t=2}h_{x,i}^{s}h_{y,j}^{t}\Vert\frac{\partial^2 w}{\partial x^{s}\partial y^{t}} \Vert_K,\\
			&\Vert(w-w^{I})_{x}\Vert_K\le C\sum_{s+t=1}h_{x,i}^{s}h_{y,j}^{t}\Vert\frac{\partial^2 w}{\partial x^{s+1}\partial y^{t}} \Vert_K,\\
			&\Vert
			(w-w^{I})_{y}\Vert_K\le C\sum_{s+t=1}h_{x,i}^{s}h_{y,j}^{t}\Vert\frac{\partial^2 w}{\partial x^{s}\partial y^{t+1}} \Vert_K,
		\end{aligned}
	\end{equation*}
	where $s$ and $t$ are non-negative integers.  
\end{lemma}
\subsection{New interpolation}
In this subsection, we design a new interpolation $\Pi S$ for the smooth part $S$. Besides, for the exponential layer $E_1$, another interpolation $\pi E_1$ is constructed based on \cite{Zhang1Liu2:2020-Optimal}.

Let $\theta_{i,j}(x,y)$ be the standard nodal basis functions with respect to the node $(x_i,y_j)$ in the finite element space $V^N$. For any $v\in C^0(\overline{\Omega})$ its Lagrange interpolation $v^I\in V^N$ is defined by
\begin{equation*}
v^{I}(x,y)=\sum_{i=0}^{N}\sum_{j=0}^{N}v(x_{i},y_{j})\theta_{i,j}(x,y).
\end{equation*}
In the subsequent analysis, we will drop $\mathrm{d}x\mathrm{d}y$ from $\int_{[]\times[]}\cdot\mathrm{d}x\mathrm{d}y$ and let $\theta_{i,j}(x,y)=\theta_{i,j}$ for simplicity.

For the solution $u$ to \eqref{model problem}, recall  \eqref{decomposition of u} and then define the interpolation $\varPi u$ to $u$ by
\begin{equation}\label{interpolation}
	\varPi u=\Pi S+\pi E_{1}+E_{2}^{I}+E_{12}^{I}.
\end{equation}
Define
\begin{equation}\label{new}
	\Pi S=\left\{\begin{aligned}
		&\mathcal{P}S\quad&&(x,y)\in[x_{\frac{N}{2}-1},x_{\frac{N}{2}}]\times[y_{\frac{N}{4}},y_{\frac{3N}{4}}],\\
		&\mathcal{L}S\quad&&(x,y)\in[x_{\frac{N}{2}-2},x_{\frac{N}{2}-1}]\times[y_{\frac{N}{4}},y_{\frac{3N}{4}}],\\
		&S^I\quad&&\text{other},
	\end{aligned}
	\right.
\end{equation}
where $\mathcal{P}S=\sum_{i=\frac{N}{2}-1}^{\frac{N}{2}}\sum_{j=\frac{N}{4}}^{\frac{3N}{4}}\mathcal{P}S(x_i,y_j)\theta_{i,j}$ satisfies
\begin{equation}\label{new interpolation}
		\begin{aligned}
			&\int_{K_{\frac{N}{2}-1,j-1}}(S-\mathcal{P}S)_x\theta_{\frac{N}{2},j}+\int_{K_{\frac{N}{2}-1,j}}(S-\mathcal{P}S)_x\theta_{\frac{N}{2},j}\\
			=&-\frac{H^2}{12}\left[\int_{y_{j-1}}^{y_{j}}\frac{\partial^2 S}{\partial x^2}\theta_{\frac{N}{2},j}(x_{\frac{N}{2}},y)\mathrm{d}y+\int_{y_j}^{y_{j+1}}\frac{\partial^2 S}{\partial x^2}\theta_{\frac{N}{2},j}(x_{\frac{N}{2}},y)\mathrm{d}y\right] \quad j=\frac{N}{4}+1,\dots,\frac{3N}{4}-1,
		\end{aligned}
\end{equation}
and
\begin{align}
	    &\mathcal{P}S(x_{\frac{N}{2}-1},y_j)=S(x_{\frac{N}{2}-1},y_j)\quad&&j=\frac{N}{4},\frac{3N}{4}\label{new interpolation 2}\\
		&\mathcal{P}S(x_{\frac{N}{2}},y_j)=S(x_{\frac{N}{2}},y_j)\quad&&j=\frac{N}{4},\frac{N}{4}+1,\dots,\frac{3N}{4}\label{new interpolation 1},
\end{align}
$\mathcal{L}S$ satisfies
\begin{equation}\label{new interpolation2}
	\left\{\begin{aligned}
		&\mathcal{L}S(x_{\frac{N}{2}-1},y_{j})=PS(x_{\frac{N}{2}-1},y_{j})\quad&&j=\frac{N}{4},\dots,\frac{3N}{4},\\
		&\mathcal{L}S(x_{\frac{N}{2}-2},y_{j})=S(x_{\frac{N}{2}-2},y_{j})\quad&&j=\frac{N}{4},\dots,\frac{3N}{4}.
	\end{aligned}
	\right.
\end{equation}
Refer to \cite{Zhang1Liu2:2020-Optimal}, we define $\pi E_1$ as
\begin{equation}\label{interpolation E}
	\pi E_1=E_1^I-\mathcal{Q}E_1+\mathcal{B}E_1,
\end{equation}
where
\begin{align}
	&\mathcal{Q}E_1=\sum_{j=0}^N E_1(x_{\frac{N}{2}-1},y_j)\theta_{\frac{N}{2}-1,j}\label{QE1},\\
	&\mathcal{B}E_1=\sum_{j=0,N} E_1(x_{\frac{N}{2}-1},y_j)\theta_{\frac{N}{2}-1,j}\label{BE1},
\end{align}
are boundary corrections to ensure the interpolation in 2D satisfies homogeneous Dirichlet boundary conditions. In addition, $E_2^I$ and $E_{12}^I$ are the standard Lagrange interpolations of $E_2$ and $E_{12}$, respectively.

\begin{remark}
For the supercloseness analysis of the smooth part, the diffusion term $\varepsilon(\nabla(S-S^I),\nabla v)$ can reach almost order $2$ on $\Omega$ using the integral identities \eqref{integral identity}; the convection term $(-b(S-S^I)_x,v)$ can achieve almost order $2$ on most regions with the method utilized in \cite{Zhang1:2003-Finite}, while on the region $[x_{\frac{N}{2}-1},x_{\frac{N}{2}+1}]\times[y_{\frac{N}{4}},y_{\frac{3N}{4}}]$, some terms cannot be better estimated with standard arguments in previous studies. Thus, we consider designing a new interpolation $\Pi S$ for sharper estimates for the tough terms: \eqref{tough 1} and \eqref{tough 2}. In the one-dimensional case \cite{zhang1Liu2:2022-Supercloseness}, intractable terms can cancel each other out with the special designed interpolation. But this idea cannot be generalized to 2D, indicating that constructing an interpolation that is both continuous and capable of optimizing the estimate of the tough terms is a very challenging task in 2D. With deep consideration, the new interpolation \eqref{new} is designed: Condition \eqref{new interpolation} is used for sharper estimates for \eqref{tough 1} and \eqref{tough 2}; Condition \eqref{new interpolation 2}, \eqref{new interpolation 1} and \eqref{new interpolation2} are employed to ensure the continuity. In addition, existence and uniqueness of $\Pi S$ is proved below.
\end{remark}

\begin{remark}
On subdomain $[x_{\frac{N}{2}-1},x_{\frac{N}{2}}]\times[y_{\frac{N}{4}},y_{\frac{3N}{4}}]$, let $S-\mathcal{P}S=S-S^I+S^I-\mathcal{P}S$. Then, for $j=\frac{N}{4}+1,\dots,\frac{3N}{4}-1$, we rewrite \eqref{new interpolation} to
\begin{equation}\label{discrete system}
\begin{aligned}
		&\frac{1}{h}\left\lbrace\int_{K_{\frac{N}{2}-1,j-1}}\left[(S^I-\mathcal{P}S)(x_{\frac{N}{2}-1},y_{j-1})\theta_{\frac{N}{2}-1,j-1}+(S^I-\mathcal{P}S)(x_{\frac{N}{2}-1},y_{j})\theta_{\frac{N}{2}-1,j}\right]_x\theta_{\frac{N}{2},j}\right.\\
		&+\left.\int_{K_{\frac{N}{2}-1,j}}\left[(S^I-\mathcal{P}S)(x_{\frac{N}{2}-1},y_{j})\theta_{\frac{N}{2}-1,j}+(S^I-\mathcal{P}S)(x_{\frac{N}{2}-1},y_{j+1})\theta_{\frac{N}{2}-1,j+1}\right]_x\theta_{\frac{N}{2},j}\right\rbrace\\
		=&-\tau_{j},
\end{aligned}
\end{equation}
where
\begin{equation*}
\tau_{j}=\ell_{1,j}+\ell_{2,j}+\ell_{3,j},
\end{equation*}
\begin{align*}
&\ell_{1,j}=:\frac{1}{h}\left\lbrace\frac{H^2}{12}\left[\int_{y_{j-1}}^{y_{j}}\frac{\partial^2 S}{\partial x^2}\theta_{\frac{N}{2},j}(x_{\frac{N}{2}},y)\mathrm{d}y+\int_{y_j}^{y_{j+1}}\frac{\partial^2 S}{\partial x^2}\theta_{\frac{N}{2},j}(x_{\frac{N}{2}},y)\mathrm{d}y\right]\right\rbrace,\\
&\ell_{2,j}=:\frac{1}{h}\left[\int_{K_{\frac{N}{2}-1,j-1}}(S-S^I)_x\theta_{\frac{N}{2},j}+\int_{K_{\frac{N}{2}-1,j}}(S-S^I)_x\theta_{\frac{N}{2},j}\right],
\end{align*}
\begin{align*}
\ell_{3,j}=:&\frac{1}{h}\left\lbrace\int_{K_{\frac{N}{2}-1,j-1}}\left[(S^I-\mathcal{P}S)(x_{\frac{N}{2}},y_{j-1})\theta_{\frac{N}{2},j-1}+(S^I-\mathcal{P}S)(x_{\frac{N}{2}},y_{j})\theta_{\frac{N}{2},j}\right]_x\theta_{\frac{N}{2},j}\right.\\
&\left.+\int_{K_{\frac{N}{2}-1,j}}\left[(S^I-\mathcal{P}S)(x_{\frac{N}{2}},y_{j})\theta_{\frac{N}{2},j}+(S^I-\mathcal{P}S)(x_{\frac{N}{2}},y_{j+1})\theta_{\frac{N}{2},j+1}\right]_x\theta_{\frac{N}{2},j}\right\rbrace.
\end{align*}
Recall \eqref{new interpolation 1}, we have
\begin{equation*}
\ell_{3,j}=0.
\end{equation*}
And we can obtain
\begin{equation*}
	\tau_{j}=\ell_{1,j}+\ell_{2,j}=\mathcal{O}(N^{-2}).
\end{equation*}
by the standard approximation theory.

For convenience, denote $(S^I-PS)(x_{\frac{N}{2}-1},y_j)=:\beta_j\,(j=\frac{N}{4}+1,\dots,\frac{3N}{4}-1)$. Then, let us look directly at the discrete system \eqref{discrete system} which we will rewrite in the form
\begin{equation}\label{linear equations}
	Db=-\tau,
\end{equation}
where
\begin{equation*}
	D=\begin{pmatrix}
		4	&  1& 0&  0&  ...&  0& 0&0 \\
		1	&  4&  1&  0&  ...&  0&  0&0 \\
		0	&  1&  4&  1&  ...&  0&  0&0 \\
		0	&  0&  1&  4&  ...&  0&  0&0 \\
		...	&  ...&  ...&  ...&  ...&  ...&  ...& \\
		0	&  0&  0& 0& ... & 1 & 4 &1 \\
		0	&  0&  0& 0 & ... & 0 & 1 &4
	\end{pmatrix}_{(\frac{N}{2}-1)\times (\frac{N}{2}-1)}\quad b=\begin{pmatrix}
		\beta_{\frac{N}{4}+1}	\\
		\beta_{\frac{N}{4}+2}	\\
		...	\\
		\beta_{\frac{3N}{4}-1}	
	\end{pmatrix}\quad \tau=\begin{pmatrix}
		\tau_{\frac{N}{4}+1}	\\
		\tau_{\frac{N}{4}+2}	\\
		...	\\
		\tau_{\frac{3N}{4}-1}	
	\end{pmatrix}.
\end{equation*}
Apparently, the determinant of tridiagonal matrix D is not zero, so the solution to the system of linear equations \eqref{linear equations} exists and is unique. This, combined with \eqref{new interpolation 1} and \eqref{new interpolation 2}, indicates that interpolation $\mathcal{P}S$ exists and is unique.

Besides, see that the max-norm of $b$ is 
\begin{equation*}
	\Vert b\Vert_{\infty}=\max_{\frac{N}{4}+1\le j\le\frac{3N}{4}-1}|\beta_j|=|\beta_{j^*}|,
\end{equation*}
According to \eqref{linear equations}, we find that	\begin{equation*}
		\left\{\begin{aligned}
			&3|\beta_{j^*}|\le|4\beta_{j^*}+\beta_{j^*+1}|=|\tau_{j^*}|=\mathcal{O}(N^{-2})\quad&&\text{ if  $j^*=\frac{N}{4}+1$},\\
			&2|\beta_{j^*}|\le|\beta_{j^*-1}+4\beta_{j^*}+\beta_{j^*+1}|=|\tau_{j^*}|=\mathcal{O}(N^{-2})\quad&&\text{ if $\frac{N}{4}+2\le j^*\le \frac{3N}{4}-2$},\\
			&3|\beta_{j^*}|\le|\beta_{j^*-1}+4\beta_{j^*}|=|\tau_{j^*}|=\mathcal{O}(N^{-2})\quad&&\text{if $j^*=\frac{3N}{4}-1$}.
			\end{aligned}
			\right.
	\end{equation*}
Thus, we can conclude that
\begin{equation}\label{dot error}
	\Vert b\Vert_{\infty}\le CN^{-2}\quad\text{i.e.}\quad |(S^I-\mathcal{P}S)(x_{\frac{N}{2}-1},y_j)|\le CN^{-2},\quad j=\frac{N}{4}+1,\dots,\frac{3N}{4}-1.
\end{equation}
\end{remark}

We need the following interpolation bound for our later analysis.
\begin{lemma}\label{max-norm error }
Let $K\in\mathcal{T}_N$, Assumptions \ref{bound} and \ref{transposition} hold. $\Pi S$ denotes the new interpolation of $S$. Then there exists a constant $C$ such that the following inequality holds:
\begin{equation*}
\Vert S-\Pi S\Vert_{L^\infty(K)}\le CN^{-2}.
\end{equation*}
\begin{proof}
Triangle inequality generates
\begin{equation}
\Vert S-\Pi S\Vert_{L^\infty(K)}\le \Vert S-S^I \Vert_{L^\infty(K)}+\Vert S^I-\Pi S\Vert_{L^\infty(K)},
\end{equation}
where
\begin{equation}\label{max-norm 1}
\Vert S-S^I\Vert_{L^\infty(K)}\le CN^{-2},
\end{equation}
is derect with Lemma \ref{interpolation error}. And using \eqref{dot error}, we can derive that
\begin{equation}\label{max-norm 2}
\Vert S^I-\Pi S\Vert_{L^\infty(K)}\le CN^{-2}.
\end{equation}
Thus we have done.
\end{proof}
\end{lemma}

The following stability of Lagrange interpolation will be frequently used later: $\forall K\in\mathcal{T}_N$ and $v\in C^1(K)$, the Lagrange interpolation $v^I$ satisfies
\begin{equation}\label{stability of Lagrange interpolation}
	\vert v^I\vert_{1,L^\infty(K)}\le C\vert v\vert_{1,L^\infty(K)}.
\end{equation}

\begin{lemma}\label{error of E}
Let Assumptions \ref{bound} and \ref{transposition} hold. $E_i^I$ denote the standard Lagrange interpolations of $E_i$, $i=1,2,12$, respectively. Then there exists a constant $C$ such that the following inequalities hold:
	\begin{align}
		&\Vert E_1-E_1^I\Vert+\Vert E_2-E_2^I\Vert+\Vert E_{12}-E_{12}^I\Vert\le CN^{-\sigma},\label{error 1}\\
		&\Vert E_1-E_1^I\Vert_\varepsilon+\Vert E_2-E_2^I\Vert_\varepsilon+\Vert E_{12}-E_{12}^I\Vert_\varepsilon\le CN^{-1},\label{error 2}\\
		&\Vert\mathcal{Q}E_1\Vert_\varepsilon+\Vert\mathcal{B}E_1\Vert_\varepsilon\le C(1+\varepsilon^{\frac{1}{4}}N^{\frac{1}{2}})N^{-\sigma}\label{error 3}.
	\end{align}
\end{lemma}
\begin{proof}
(a) Consider $\Vert E_1-E_1^I\Vert_\varepsilon$. We split $\Vert E_1-E_1^I\Vert$ into 
	\begin{align*}
		\Vert E_1-E_1^I\Vert^2=\Vert E_1-E_1^I\Vert_{[0,x_{\frac{N}{2}-1}]\times[0,1]}^2+\Vert E_1-E_1^I\Vert_{\Omega_0}^2+\Vert E_1-E_1^I\Vert_{[x_{\frac{N}{2}},1]\times[0,1]}^2.
	\end{align*}
	Use Lemma \ref{interpolation error}, Assumption \ref{bound} and Lemma \ref{mesh step} to derive
	\begin{equation}\label{E1 1}
		\begin{aligned}
			\Vert E_1-E_1^I\Vert_{[0,x_{\frac{N}{2}-1}]\times[0,1]}^2&\le C\sum_{i=0}^{\frac{N}{2}-2}\sum_{j=0}^{N-1}\sum_{s+t=2}h_{x,i}^{2s}h_{y,j}^{2t}\Vert\frac{\partial^2 E_1}{\partial x^s\partial y^t}\Vert_{K_{i,j}}^2\\
			&\le C\sum_{i=0}^{\frac{N}{2}-2}\sum_{j=0}^{N-1}\sum_{s+t=2}h_{x,i}^{2s}h_{y,j}^{2t}(\varepsilon^{-s}e^{-\frac{\beta x_i}{\varepsilon}})^2h_{x,i}h_{y,j}\\
			&\le C\varepsilon N^{-4}.
		\end{aligned}
	\end{equation}
Triangular inequality, Assumption \ref{bound}, Lemma \ref{mesh step} and \eqref{stability of Lagrange interpolation} yield
	\begin{equation*}
		\Vert E_1-E_1^I\Vert_{\Omega_0}^2\le\Vert E_1\Vert_{\Omega_0}^2+\Vert E_1^I\Vert_{\Omega_0}^2,
	\end{equation*}
	where
	\begin{align*}
		&\Vert E_1\Vert_{\Omega_0}^2\le C\int_{\Omega_0}(e^{-\frac{\beta x}{\varepsilon}})^2\le C\varepsilon N^{-2\sigma},\\
		&\Vert E_1^I\Vert_{\Omega_0}^2\le C\int_{\Omega_0}(N^{-\sigma})^2\le CN^{-2\sigma-1}.
	\end{align*}
	Thus,
	\begin{equation}\label{E1 2}
		\Vert E_1-E_1^I\Vert_{\Omega_0}^2\le CN^{-2\sigma-1}.
	\end{equation}
	Follow the same argument as for $\Vert E_1-E_1^I\Vert_{\Omega_0}^2$, we have
	\begin{equation}\label{E1 3}
		\Vert E_1-E_1^I\Vert_{[x_{\frac{N}{2}},1]\times[0,1]}^2\le C\varepsilon^{2\sigma}.
	\end{equation}
	
	$\Vert (E_1-E_1^I)_x\Vert$ can be decomposed as
	\begin{equation*}
		\Vert (E_1-E_1^I)_x\Vert^2=\Vert (E_1-E_1^I)_x\Vert_{\Omega_x\cup\Omega_{xy}}^2+\Vert (E_1-E_1^I)_x\Vert_{\Omega_s\cup\Omega_y}^2.
	\end{equation*}
	Similar to \eqref{E1 1}, one has
	\begin{equation}
		\Vert (E_1-E_1^I)_x\Vert_{\Omega_x\cup\Omega_{xy}}^2\le C\varepsilon^{-1}N^{-2}.
	\end{equation}
	Triangular inequality, inverse inequality, Assumption \ref{bound}, Lemma \ref{mesh step} and \eqref{stability of Lagrange interpolation} yield
	\begin{equation*}
		\Vert (E_1-E_1^I)_x\Vert_{\Omega_s\cup\Omega_y}^2\le C(\Vert (E_1)_x\Vert_{\Omega_s\cup\Omega_y}^2+\Vert(E_1^I)_x\Vert_{\Omega_0}^2+\Vert(E_1^I)_x\Vert_{[x_{\frac{N}{2}},1]\times[0,1]}^2),
	\end{equation*}
	where
	\begin{align*}
		&\Vert (E_1)_x\Vert_{\Omega_s\cup\Omega_y}^2\le C\varepsilon^{-1}N^{-2\sigma},\\
		&\Vert(E_1^I)_x\Vert_{\Omega_0}^2\le Ch_{x,\frac{N}{2}-1}^{-2}N^{-2\sigma}h_{x,\frac{N}{2}-1}\le C\varepsilon^{-1}N^{-2\sigma},\\
		&\Vert(E_1^I)_x\Vert_{[x_{\frac{N}{2}},1]\times[0,1]}^2\le C\varepsilon^{2\sigma-2}.
	\end{align*}
	Therefore,
	\begin{equation}
		\Vert (E_1-E_1^I)_x\Vert_{\Omega_s\cup\Omega_y}^2\le C\varepsilon^{-1}N^{-2\sigma}.
	\end{equation}
	
	The argument for $\Vert (E_1-E_1^I)_y\Vert$ is similar to $\Vert (E_1-E_1^I)_x\Vert$, hence
	\begin{align}
		&\Vert (E_1-E_1^I)_y\Vert_{\Omega_x\cup\Omega_{xy}}^2\le C\varepsilon N^{-2}.\\
		&\Vert (E_1-E_1^I)_y\Vert_{\Omega_s\cup\Omega_y}^2\le CN^{-2\sigma}.
	\end{align}
	
	(b) Consider $\Vert E_2-E_2^I\Vert_\varepsilon$ and $\Vert E_{12}-E_{12}^I\Vert_\varepsilon$. Their proofs are analogous to $\Vert E_1-E_1^I\Vert_\varepsilon$, therefore are omitted. Here we just give the main conclusions that will be used in our later analysis.
	\begin{align}
		&\Vert E_2-E_2^I\Vert_{\Omega_y\cup\Omega_{xy}}^2\le C\varepsilon^{\frac{1}{2}}N^{-4},\\
		&\Vert E_2-E_2^I\Vert_{\Omega_x\cup\Omega_s}^2\le CN^{-2\sigma},\\
		&\Vert (E_2-E_2^I)_x\Vert_{\Omega_y\cup\Omega_{xy}}^2\le C\varepsilon^{\frac{1}{2}}N^{-2},\\
		&\Vert (E_2-E_2^I)_x\Vert_{\Omega_x\cup\Omega_s}^2\le CN^{-2\sigma}\label{E2 4},\\
		&\Vert (E_2-E_2^I)_y\Vert_{\Omega_y\cup\Omega_{xy}}^2\le C\varepsilon^{-\frac{1}{2}}N^{-2},\\
		&\Vert (E_2-E_2^I)_y\Vert_{\Omega_x\cup\Omega_s}^2\le C\varepsilon^{-\frac{1}{2}}N^{-2\sigma},
	\end{align}
	and
	\begin{align}
		&\Vert E_{12}-E_{12}^I\Vert_{\Omega_{xy}}^2\le C\varepsilon^{\frac{3}{2}}N^{-4},\label{E12 1}\\
		&\Vert E_{12}-E_{12}^I\Vert_{\Omega\textbackslash\Omega_{xy}}^2\le CN^{-2\sigma},\\
		&\Vert (E_{12}-E_{12}^I)_x\Vert_{\Omega_{xy}}^2\le C\varepsilon^{-\frac{1}{2}}N^{-2},\\
		&\Vert (E_{12}-E_{12}^I)_x\Vert_{\Omega_x\cup\Omega_y}^2\le C\varepsilon^{-1}N^{-2\sigma},\\
		&\Vert (E_{12}-E_{12}^I)_x\Vert_{\Omega_s}^2\le C\varepsilon^{-1}N^{-4\sigma},\\
		&\Vert (E_{12}-E_{12}^I)_y\Vert_{\Omega_{xy}}^2\le C\varepsilon^{\frac{1}{2}}N^{-2},\\
		&\Vert (E_{12}-E_{12}^I)_y\Vert_{\Omega_x\cup\Omega_y}^2\le CN^{-2\sigma},\\
		&\Vert (E_{12}-E_{12}^I)_y\Vert_{\Omega_s}^2\le C\varepsilon^{-1}N^{-4\sigma}.
	\end{align}
The above inequalities give the estimates for \eqref{error 1} and \eqref{error 2}.

(c) Consider $\Vert\mathcal{Q}E_1\Vert_\varepsilon$ and $\Vert\mathcal{B}E_1\Vert_\varepsilon$. \eqref{QE1}, \eqref{BE1}, Lemma \ref{mesh step} and inverse inequality generate
	\begin{equation}
		\begin{aligned}
			\Vert\mathcal{Q}E_1\Vert_\varepsilon^2&\le CN^{-2\sigma}\sum_{j=0}^{N}\Vert\theta_{\frac{N}{2}-1,j}\Vert_\varepsilon^2\\
			&\le CN^{-2\sigma}\sum_{j=0}^{N}(\varepsilon h_{x,\frac{N}{2}-1}^{-2} h_{x,\frac{N}{2}-1}h_{y,j}+\varepsilon h_{y,j}^{-2} h_{x,\frac{N}{2}-1}h_{y,j}+h_{x,\frac{N}{2}-1}h_{y,j})\\
			&\le CN^{-2\sigma},
		\end{aligned}
	\end{equation}
	and similarly,
	\begin{equation}
		\Vert\mathcal{B}E_1\Vert_\varepsilon^2\le C\varepsilon^{\frac{1}{2}}N^{-2\sigma}.
	\end{equation}
	Thus we have done.
\end{proof}

\section{Supercloseness property on a Bakhvalov-type mesh}\label{sec. 4}
In this section, supercloseness $\Vert u^N-\varPi u\Vert_\varepsilon$ will be analyzed in detail. \eqref{variation} gives
\begin{equation*}
C\Vert v\Vert_{\varepsilon}^{2}\le a(v,v)=a(u-\varPi u,v)= \varepsilon(\nabla(u-\varPi u),\nabla v)+(-b(u-\varPi u)_{x},v)+(c(u-\varPi u),v).
\end{equation*}

\begin{lemma}\label{all term}
Let $u$ be the solution to \eqref{model problem} satisfying Assumptions \ref{bound} and \ref{transposition}, $\varPi u$ be the new interpolation of $u$. For any $v\in V^{N}$, one has
\begin{equation*}
\vert\varepsilon (\nabla(u-\varPi u),\nabla v)+(c(u-\varPi u),v)\vert \le CN^{-2}\Vert v\Vert_{\varepsilon}.
\end{equation*}
\begin{proof}
\begin{align*}
\varepsilon (\nabla(u-\varPi u),\nabla v)=&\varepsilon(\nabla(E_1+E_2+E_{12}-\pi E_1-E_2^I-E_{12}^I),\nabla v)+\varepsilon(\nabla(S-\Pi S),\nabla v).
\end{align*}

(a) Consider $\varepsilon(\nabla(E_1+E_2+E_{12}-\pi E_1-E_2^I-E_{12}^I),\nabla v)$. The following estimates are straightward:
\begin{align}
&|\varepsilon\int_{\Omega_y\cup\Omega_s}(E_1-\pi E_1)_x v_x+\varepsilon\int_\Omega(E_1-\pi E_1)_y v_y|\le CN^{-\sigma}\Vert v\Vert_\varepsilon,\label{diffusion 1}\\
&|\varepsilon\int_\Omega(E_2-E_2^I)_xv_x+\varepsilon\int_{\Omega_x\cup\Omega_s}(E_2-E_2^I)_yv_y|\le C\varepsilon^{\frac{1}{4}}N^{-\sigma}\Vert v\Vert_\varepsilon,\\
&|\varepsilon\int_{\Omega\textbackslash\Omega_{xy}}(E_{12}-E_{12}^I)_xv_x+\varepsilon\int_{\Omega\textbackslash\Omega_{xy}}(E_{12}-E_{12}^I)_yv_y|\le CN^{-\sigma}\Vert v\Vert_\varepsilon\label{diffusion 3},
\end{align}
where we have used some conclusions in Lemma \ref{error of E}.

To analyze the remaining terms, we introduce two integral identities from \cite{Lin1Yan2:1996-Construction} that have been detailed demonstrated in \cite[Th 4.3.]{Zhang1:2003-Finite}:
\begin{subequations}\label{integral identity}
\begin{align}
		&\int_K\frac{\partial}{\partial x}(w-w^I)\frac{\partial v}{\partial x}=\int_K\frac{\partial^3 w}{\partial x\partial y^2}F(y)(\frac{\partial v}{\partial x}-\frac{2}{3}(y-y_K)\frac{\partial^2 v}{\partial x\partial y}),\label{integral identity 1}\\
	&\int_K\frac{\partial}{\partial y}(w-w^I)\frac{\partial v}{\partial y}=\int_K\frac{\partial^3 w}{\partial x^2\partial y}E(x)(\frac{\partial v}{\partial y}-\frac{2}{3}(x-x_K)\frac{\partial^2 v}{\partial x\partial y}),\label{integral identity 2}
\end{align}
\end{subequations}
where
\begin{align*}
	&F(y)=\frac{(y-y_K)^2-\hbar_{K}^2}{2},\\
	&E(x)=\frac{(x-x_K)^{2}-h_{K}^2}{2},
\end{align*}
$K\in \mathcal{T}_N$, $h_K$ denotes half the length of $K$ in the $x$ direction, $\hbar_K$ denotes half the width of $K$ in the y direction, $(x_K,y_K)$ denotes the center of $K$, and $K$ can be denoted as $(x_K-h_K,x_K+h_K)\times(y_K-\hbar_K,y_K+\hbar_K)$. 

For term $\varepsilon\int_{\Omega_x\cup\Omega_{xy}}(E_1-\pi E_1)_xv_x$, we first split it into
\begin{equation*}
\varepsilon\int_{\Omega_x\cup\Omega_{xy}}(E_1-\pi E_1)_xv_x=\varepsilon\int_{\Omega_x\cup\Omega_{xy}}(E_1- E_1^I)_xv_x	+\varepsilon\int_{\Omega_x\cup\Omega_{xy}}(E_1^I-\pi E_1)_xv_x.
\end{equation*}
Use \eqref{integral identity 1} to derive
\begin{equation*}
|\varepsilon\int_{K\subset\Omega_x\cup\Omega_{xy}}( E_{1}-E_1^I)_{x}v_{x}|\le C\int_K e^{-\frac{\beta x}{\varepsilon}}|F(y)||\frac{\partial v}{\partial x}|\le CN^{-2}\Vert e^{-\frac{\beta x}{\varepsilon}}\Vert_K\Vert\frac{\partial v}{\partial x}\Vert_K,
\end{equation*}
summing over $K\subset\Omega_x\cup\Omega_{xy}$ and applying Cauchy-Schwarz inequality, we get
\begin{equation}\label{diffusion 4}
|\varepsilon\int_{\Omega_x\cup\Omega_{xy}}( E_{1}- E_1^I)_{x}v_{x}|\le CN^{-2}\Vert v\Vert_\varepsilon.
\end{equation}
\eqref{interpolation E}-\eqref{BE1} combined with inverse inequality yield
\begin{equation}\label{diffusion 5}
	\begin{aligned}
|\varepsilon\int_{\Omega_x\cup\Omega_{xy}}(E_1^I-\pi E_1)_xv_x|&=|\varepsilon\int_{\Omega_x\cup\Omega_{xy}}\sum_{j=1}^{N-1}(E_1(x_{\frac{N}{2}-1},y_j)\theta_{\frac{N}{2}-1,j})_xv_x|\\
&\le C\varepsilon N^{-\sigma}\sum_{j=1}^{N-1}\Vert(\theta_{\frac{N}{2}-1,j})_x\Vert_{\Omega_x\cup\Omega_{xy}}\Vert v_x\Vert_{\Omega_x\cup\Omega_{xy}}\\
&\le C\varepsilon N^{-\sigma}\sum_{j=0}^{N-1}h_{x,\frac{N}{2}-2}^{-1}h_{x,\frac{N}{2}-2}^{\frac{1}{2}}h_{y,j}^{\frac{1}{2}}\varepsilon^{-\frac{1}{2}}\Vert v\Vert_{\Omega_x\cup\Omega_{xy}}\\
&\le CN^{-\sigma+\frac{1}{2}}\Vert v\Vert_\varepsilon.
\end{aligned}
\end{equation}
Similar to $\varepsilon\int_{\Omega_x\cup\Omega_{xy}}(E_1-E_1^I)_xv_x$, except that \eqref{integral identity 2} is used, we can obtain
\begin{equation}\label{diffusion 6}
|\varepsilon\int_{\Omega_{y}\cup\Omega_{xy}}(E_{2}-E_{2}^I)_{y}v_{y}|\le C\varepsilon^{\frac{1}{4}}N^{-2}\Vert v\Vert_\varepsilon.
\end{equation}
Lemma \ref{mesh step} combined with \eqref{integral identity 1} generates
\begin{align*}
|\varepsilon\int_{K\subset\Omega_{xy}}(E_{12}-E_{12}^I)_{x}v_{x}|&\le C\varepsilon^{-1}\int_K\left[e^{-\frac{\beta x}{\varepsilon}}(e^{-\frac{ y}{\sqrt\varepsilon}}+e^{-\frac{1- y}{\sqrt\varepsilon}})\right]|F(y)||\frac{\partial v}{\partial x}|\\
&\le C\varepsilon^{-1}\left[h_{y,j}^2\max(e^{-\frac{ y}{\sqrt\varepsilon}}+e^{-\frac{1- y}{\sqrt\varepsilon}})\right]\Vert e^{-\frac{\beta x}{\varepsilon}}\Vert_K\Vert\frac{\partial v}{\partial x}\Vert_K\\
&\le C\varepsilon^{-1}(\sqrt\varepsilon N^{-1})^2\Vert e^{-\frac{\beta x}{\varepsilon}}\Vert_K\Vert\frac{\partial v}{\partial x}\Vert_K,
\end{align*}
then summing up:
\begin{equation}\label{diffusion 7}
|\varepsilon\int_{\Omega_{xy}}(E_{12}^{I}-E_{12})_{x}v_{x}|\le CN^{-2}\Vert v\Vert_\varepsilon.
\end{equation}
Similar to $\varepsilon\int_{K\subset\Omega_{xy}}(E_{12}-E_{12}^I)_{x}v_{x}$, just substituting \eqref{integral identity 1} to \eqref{integral identity 2}, we have
\begin{equation}\label{diffusion 8}
|\varepsilon\int_{\Omega_{xy}}(E_{12}^{I}-E_{12})_{y}v_{y}|\le C\varepsilon^{\frac{1}{4}}N^{-2}\Vert v\Vert_\varepsilon.
\end{equation}                                                 
Collect \eqref{diffusion 1}-\eqref{diffusion 3} and \eqref{diffusion 4}-\eqref{diffusion 8}, one has
\begin{equation}\label{4.10}
|\varepsilon(\nabla(E_1+E_2+E_{12}-\pi E_1-E_2^I-E_{12}^I),\nabla v)|\le CN^{-2}\Vert v\Vert_\varepsilon.
\end{equation}

(b) Consider $\varepsilon(\nabla(S-\Pi S),\nabla v)$. It can be split into
\begin{align*}
	&\varepsilon(\nabla(S-\Pi S),\nabla v)\\
	=&\varepsilon\int_{\Omega\textbackslash([x_{\frac{N}{2}-2},x_{\frac{N}{2}}]\times[y_{\frac{N}{4}},y_{\frac{3N}{4}}])}(\nabla(S-S^I),\nabla v)+\varepsilon\int_{[x_{\frac{N}{2}-2},x_{\frac{N}{2}}]\times[y_{\frac{N}{4}},y_{\frac{3N}{4}}]}(\nabla(S-\Pi S),\nabla v).
\end{align*}
 Use the two integral identities \eqref{integral identity 1} and \eqref{integral identity 2} to obtain
	\begin{equation*}
		|\varepsilon\int_{K}(\nabla(S-S^I),\nabla v)|\le C\varepsilon N^{-3}\Vert\nabla v\Vert_K,
	\end{equation*}
summing over $K\subset\Omega\textbackslash([x_{\frac{N}{2}-2},x_{\frac{N}{2}}]\times[y_{\frac{N}{4}},y_{\frac{3N}{4}}])$ and taking use of Cauchy-Schwarz inequality, one has
\begin{equation}\label{4.11}
|\varepsilon\int_{\Omega\textbackslash([x_{\frac{N}{2}-2},x_{\frac{N}{2}}]\times[y_{\frac{N}{4}},y_{\frac{3N}{4}}])}(\nabla(S-S^I),\nabla v)|\le C\varepsilon^{\frac{1}{2}}N^{-2}\Vert v\Vert_\varepsilon.
\end{equation}
Let $S-\Pi S=S-S^I+S^I-\Pi S$, 
\begin{equation}
		|\varepsilon\int_{[x_{\frac{N}{2}-2},x_{\frac{N}{2}}]\times[y_{\frac{N}{4}},y_{\frac{3N}{4}}]}(\nabla(S- S^I),\nabla v)|\le C\varepsilon^{\frac{1}{2}}N^{-\frac{3}{2}}\Vert v\Vert_\varepsilon.
\end{equation}
is straightforward with Lemma \ref{interpolation error}. Apply \eqref{max-norm 2}, inverse inequality and Cauchy-Schwarz inequality, we have
\begin{equation}\label{4.13}
\begin{aligned}
	&|\varepsilon\int_{[x_{\frac{N}{2}-2},x_{\frac{N}{2}}]\times[y_{\frac{N}{4}},y_{\frac{3N}{4}}]}(\nabla(S^I- \Pi S),\nabla v)|\\
\le& C\varepsilon\sum_{i=\frac{N}{2}-2}^{\frac{N}{2}-1}\sum_{j=\frac{N}{4}}^{\frac{3N}{4}-1}h_{x,i}^{-1}\Vert S^I-\Pi S\Vert_{L^\infty(K_{i,j})}(h_{x,i}h_{y,j})^{\frac{1}{2}}\Vert\nabla v\Vert_{K_{i,j}}\\
\le& CN^{-2}\Vert v\Vert_\varepsilon.
\end{aligned}
\end{equation}
Collect \eqref{4.11}-\eqref{4.13}, one can obtain
\begin{equation}\label{4.14}
|\varepsilon(\nabla(S-\Pi S),\nabla v)|\le CN^{-2}\Vert v\Vert_\varepsilon.
\end{equation}

(c) Consider $(c(u-\varPi u),v)$. Recall Lemma \ref{max-norm error } and Lemma \ref{error of E}, we have
\begin{equation}\label{4.15}
|(c(u-\varPi u),v)|\le CN^{-2}\Vert v\Vert_\varepsilon.
\end{equation}

Then, with \eqref{4.10}, \eqref{4.14} and \eqref{4.15}, we establish the assertion for Lemma \ref{all term}.
\end{proof}
\end{lemma}

\begin{lemma}\label{term E}
	Let Assumptions \ref{bound} and \ref{transposition} hold, $\pi E_1$ be the interpolation of $E_1$ defined in \eqref{interpolation E}, $E_2^I$ and $E_{12}^I$ be the standard Lagrange interpolations of $E_2$ and $E_{12}$, respectively. For any $v\in V^{N}$, one has
	\begin{equation*}
		|(-b(E_1+E_2+E_{12}-\pi E_1-E_2^I-E_{12}^I)_x,v)|\le C\varepsilon^{\frac{1}{4}}N^{-\frac{3}{2}}\Vert v\Vert_\varepsilon+CN^{-2}\Vert v\Vert_\varepsilon.
	\end{equation*}
\begin{proof}
Green's formula generates
\begin{align*}
	&(-b(E_1+E_2+E_{12}-\pi E_1-E_2^I-E_{12}^I)_x,v)\\
	=&\int_{\Omega}b(E_{1}-\pi E_1)v_{x}+\int_{\Omega}b(E_{12}-E_{12}^{I})v_{x}+\int_{\Omega}b_{x}\left[(E_{1}-\pi E_{1})+(E_{12}-E_{12}^{I})\right]v-\int_{\Omega}b(E_{2}-E_{2}^{I})_{x}v.
\end{align*}

(a) Consider $\int_{\Omega}b(E_{1}-\pi E_1)v_{x}$. It can be decomposed as
\begin{align*}
	\int_{\Omega}b(E_{1}-\pi E_1)v_{x}=\int_{\Omega\textbackslash\Omega_0}b(E_1-E_1^I)v_x+\int_{\Omega\textbackslash\Omega_0}b\mathcal{Q}E_1v_x-\int_\Omega b\mathcal{B}E_1v_x+\int_{\Omega_0}b\left[E_1-(E_1^I-\mathcal{Q}E_1)\right]v_x.
\end{align*}	
Use H\"{o}lder inequality, \eqref{E1 1} and \eqref{E1 3} to obtain
\begin{equation}\label{4.16}
|\int_{\Omega\textbackslash\Omega_0}b(E_1-E_1^I)v_x|\le C(N^{-2}+C\varepsilon^{\sigma-\frac{1}{2}})\Vert v\Vert_\varepsilon.
\end{equation}
From \eqref{QE1}, we know that $\mathcal{Q}E_1|_{\Omega\textbackslash\Omega_0}=\sum_{j=0}^N E_1(x_{\frac{N}{2}-1},y_j)\theta_{\frac{N}{2}-1,j}$, then
\begin{equation}
\begin{aligned}
	|\int_{\Omega\textbackslash\Omega_0}b\mathcal{Q}E_1v_x|&\le C\Vert \mathcal{Q}E_1\Vert_{\Omega\textbackslash\Omega_0}\Vert v_x\Vert_{\Omega\textbackslash\Omega_0}\\
	&\le CN^{-\sigma}\sum_{j=0}^N \Vert\theta_{\frac{N}{2}-1,j}\Vert_{\Omega\textbackslash\Omega_0}\Vert v_x\Vert_{\Omega\textbackslash\Omega_0}\\
	&\le CN^{-\sigma}\sum_{j=0}^{N-1} (h_{x,\frac{N}{2}-2}h_{y,j})^{\frac{1}{2}}\Vert v_x\Vert_{\Omega\textbackslash\Omega_0}\\
	& \le CN^{-\sigma+\frac{1}{2}}\Vert v\Vert_\varepsilon.
\end{aligned}
\end{equation}	
By \eqref{BE1}, one has
\begin{equation}
\begin{aligned}
	|-\int_\Omega b\mathcal{B}E_1v_x|&\le CN^{-\sigma}\sum_{j=0,N}\Vert\theta_{\frac{N}{2}-1,j}\Vert_\Omega\Vert v_x\Vert_\Omega\\
	&\le CN^{-\sigma}\sum_{j=0,N-1}(h_{x,\frac{N}{2}-1}h_{y,j})^{\frac{1}{2}}\varepsilon^{-\frac{1}{2}}\Vert v\Vert_\varepsilon\\
	&\le CN^{-\sigma}(\varepsilon^{\frac{1}{2}}N^{-\frac{1}{2}}\sqrt\varepsilon)^{\frac{1}{2}}\varepsilon^{-\frac{1}{2}}\Vert v\Vert_\varepsilon\\
	&\le CN^{-\sigma-\frac{1}{4}}\Vert v\Vert_\varepsilon,
\end{aligned}
\end{equation}
here we have used $h_{x,\frac{N}{2}-1}\le C\varepsilon^{\frac{1}{2}}N^{-\frac{1}{2}}$ from Lemma \ref{special step}.

For $\int_{\Omega_0}b\left[E_1-(E_1^I-\mathcal{Q}E_1)\right]v_x$, according to \eqref{QE1} and triangle inequality, we can see that
\begin{equation*}
	\Vert E_1-(E_1^I-\mathcal{Q}E_1)\Vert_{\Omega_0}^2\le \Vert E_1\Vert_{\Omega_0}^2+\Vert E_1^I-\mathcal{Q}E_1\Vert_{\Omega_0}^2,\quad E_1^I-\mathcal{Q}E_1|_{\Omega_0}=\sum_{j=0}^N E_1(x_{\frac{N}{2}},y_j)\theta_{\frac{N}{2},j},
\end{equation*}
where
\begin{align*}
	&\Vert E_1\Vert_{\Omega_0}^2\le C\varepsilon N^{-2\sigma},\\
	&\Vert E_1^I-\mathcal{Q}E_1\Vert_{\Omega_0}^2\le C\varepsilon^{2\sigma}\sum_{j=0}^N\Vert \theta_{\frac{N}{2},j}\Vert_{\Omega_0}^2\le C\varepsilon^{2\sigma}\sum_{j=0}^{N-1}(h_{x,\frac{N}{2}-1}h_{y,j})\le C\varepsilon^{2\sigma}N^{-1}.
\end{align*}
Thus,
\begin{equation}\label{4.19}
	|\int_{\Omega_0}b\left[E_1-(E_1^I-\mathcal{Q}E_1)\right]v_x|\le CN^{-\sigma}\Vert v\Vert_\varepsilon.
\end{equation}	
Collect \eqref{4.16}-\eqref{4.19}, we get
\begin{equation}\label{4.20}
|\int_{\Omega}b(E_{1}-\pi E_1)v_{x}|\le CN^{-2}\Vert v\Vert_\varepsilon.
\end{equation}

(b) Consider $\int_{\Omega}b(E_{12}-E_{12}^{I})v_{x}$. It will be separated into the six cases $\Omega_{xy}$, $\Omega_x$, $[x_{\frac{N}{2}-1},x_{\frac{N}{2}}]\times([0,y_{\frac{N}{4}-1}]\cup[y_{\frac{3N}{4}+1},1])$, $[x_{\frac{N}{2}-1},x_{\frac{N}{2}}]\times([y_{\frac{N}{4}-1},y_{\frac{N}{4}}]\cup[y_{\frac{3N}{4}},y_{\frac{3N}{4}+1}])$, $[x_{\frac{N}{2}-1},x_{\frac{N}{2}}]\times[y_{\frac{N}{4}},y_{\frac{3N}{4}}]$, and $[x_{\frac{N}{2}},1]\times[0,1]$.
\begin{equation*}
|\int_{\Omega_{xy}}b(E_{12}-E_{12}^{I})v_{x}|\le C\varepsilon^{\frac{1}{4}}N^{-2}\Vert v\Vert_\varepsilon,
\end{equation*}
can be deduced by \eqref{E12 1}. Furthermore, use H\"{o}lder inequality, Lemma \ref{mesh step}, Lemma \ref{special step} and \eqref{stability of Lagrange interpolation} to derive
\begin{align*}
	&|\int_{\Omega_{x}}b(E_{12}-E_{12}^{I})v_{x}|\le C\Vert E_{12}\Vert_{L^\infty(\Omega_x)}(\text{meas}\Omega_x)^{\frac{1}{2}}\Vert v_x\Vert_{\Omega_x}\le CN^{-\sigma}\ln^{\frac{1}{2}}N\Vert v\Vert_\varepsilon,\\
	&|\int_{[x_{\frac{N}{2}-1},x_{\frac{N}{2}}]\times([0,y_{\frac{N}{4}-1}]\cup[y_{\frac{3N}{4}+1},1])}b(E_{12}-E_{12}^{I})v_{x}|\le CN^{-\sigma}(\varepsilon^{\frac{3}{4}}N^{-\frac{1}{4}}\sqrt\varepsilon\ln{N})^{\frac{1}{2}}\Vert v_x\Vert_\varepsilon\le C\varepsilon^{\frac{1}{8}}N^{-\sigma-\frac{1}{8}}\ln^{\frac{1}{2}}{N}\Vert v\Vert_\varepsilon,\\
	&|\int_{[x_{\frac{N}{2}-1},x_{\frac{N}{2}}]\times([y_{\frac{N}{4}-1},y_{\frac{N}{4}}]\cup[y_{\frac{3N}{4}},y_{\frac{3N}{4}+1}])}b(E_{12}-E_{12}^{I})v_{x}|\le CN^{-\sigma}(\varepsilon^{\frac{3}{4}}N^{-\frac{1}{4}}(\sqrt\varepsilon)^{\frac{1}{2}}N^{-\frac{1}{2}})^{\frac{1}{2}}\Vert v_x\Vert_\varepsilon\le CN^{-\sigma-\frac{3}{8}}\Vert v\Vert_\varepsilon,\\
	&|\int_{[x_{\frac{N}{2}-1},x_{\frac{N}{2}}]\times[y_{\frac{N}{4}},y_{\frac{3N}{4}}]}b(E_{12}-E_{12}^{I})v_{x}|\le C\varepsilon^{\sigma-\frac{1}{2}}N^{-\frac{1}{2}-\sigma}\Vert v\Vert_\varepsilon,\\
	&|\int_{[x_{\frac{N}{2}},1]\times[0,1]}b(E_{12}-E_{12}^{I})v_{x}|\le C\varepsilon^{\sigma-\frac{1}{2}}\Vert v\Vert_\varepsilon,
\end{align*}
where $h_{x,\frac{N}{2}-1}\le C\varepsilon^{\frac{3}{4}}N^{-\frac{1}{4}}$ and $h_{y,\frac{N}{4}-1}=h_{y,\frac{3N}{4}}\le C(\sqrt\varepsilon)^{\frac{1}{2}}N^{-\frac{1}{2}}$ are used and they can be deduced by Lemma \ref{special step}.

In conclusion,
\begin{equation}
	|\int_{\Omega}b(E_{12}-E_{12}^{I})v_{x}|\le CN^{-\sigma}\ln^{\frac{1}{2}}N\Vert v\Vert_\varepsilon+C\varepsilon^{\frac{1}{4}}N^{-2}\Vert v\Vert_\varepsilon.
\end{equation}

(c) Consider $\int_{\Omega}b_{x}\left[(E_{1}-\pi E_{1})+(E_{12}-E_{12}^{I})\right]v$. Lemma \ref{error of E} generates
\begin{equation}
	|\int_{\Omega}b_{x}\left[(E_{1}-\pi E_{1})+(E_{12}-E_{12}^{I})\right]v|\le CN^{-\sigma}\Vert v\Vert_\varepsilon.
\end{equation}

(d)Consider $\int_{\Omega}b(E_{2}-E_{2}^{I})_{x}v$. We separate the discussion into the cases of $\Omega_x\cup\Omega_s$ and $\Omega_{xy}\cup\Omega_y$. Recall \eqref{E2 4}, one has
\begin{align*}
	|\int_{\Omega_{x}\cup\Omega_s}b(E_{2}-E_{2}^{I})_{x}v|\le CN^{-\sigma}\Vert v\Vert_\varepsilon.
\end{align*}
Green's formula generates
\begin{align*}
	\int_{\Omega_{xy}\cup\Omega_y}b(E_{2}-E_{2}^{I})_{x}v&=-\int_{\Omega_{xy}\cup\Omega_y}(b_xv+bv_x)(E_2-E_2^I).
\end{align*}
Then, we split $\Omega_{xy}\cup\Omega_y$ into $\Omega_{xy}$, $[x_{\frac{N}{2}-1},x_{\frac{N}{2}}]\times([0,y_{\frac{N}{4}-1}]\cup[y_{\frac{3N}{4}+1},1])$ and $[x_{\frac{N}{2}},1]\times([0,y_{\frac{N}{4}-1}]\cup[y_{\frac{3N}{4}+1},1])$. In a similar way to \eqref{E1 1}, one can obtain
\begin{equation}
|-\int_{\Omega_{xy}}(b_xv+bv_x)(E_2-E_2^I)|\le C\varepsilon^{\frac{1}{4}}N^{-\frac{3}{2}}\Vert v\Vert_\varepsilon.
\end{equation}

Assumption \ref{bound}, Lemma \ref{mesh step}, Lemma \ref{special step}, H\"{o}lder inequality, triangle inequality and \eqref{stability of Lagrange interpolation} generate
\begin{align*}
	&|\int_{[x_{\frac{N}{2}-1},x_{\frac{N}{2}}]\times([0,y_{\frac{N}{4}-1}]\cup[y_{\frac{3N}{4}+1},1])}(b_xv+bv_x)(E_2-E_2^I)|\le CN^{-\sigma-\frac{1}{4}}\ln^{\frac{1}{2}}N\Vert v\Vert_\varepsilon,\\
	&|\int_{[x_{\frac{N}{2}},1]\times([0,y_{\frac{N}{4}-1}]\cup[y_{\frac{3N}{4}+1},1])}(b_xv+bv_x)(E_2-E_2^I)|\le C\varepsilon^{\sigma-\frac{1}{4}}\ln^{\frac{1}{2}}N\Vert v\Vert_\varepsilon.
\end{align*}
In conclusion,
\begin{equation}\label{4.23}
	|\int_{\Omega}b(E_{2}-E_{2}^{I})_{x}v|\le C\varepsilon^{\frac{1}{4}}N^{-\frac{3}{2}}\Vert v\Vert_\varepsilon.
\end{equation}

From \eqref{4.20}-\eqref{4.23}, we give the proof of Lemma \ref{term E}.
\end{proof}
\end{lemma}

\begin{lemma}\label{term S}
	Let Assumptions \ref{bound} and \ref{transposition} hold, $\Pi S$ be the new interpolation of $S$. For any $v\in V^{N}$, one has
	\begin{equation*}
		|(-b(S-\Pi S)_x,v)|\le C\varepsilon^{\frac{1}{4}}N^{-\frac{3}{2}}\ln^{\frac{1}{2}}{N}\Vert v\Vert_\varepsilon+CN^{-2}\ln^{\frac{1}{2}}N\Vert v\Vert_\varepsilon.
	\end{equation*}
\end{lemma}	
\begin{proof}
	We decompose $(-b(S-\Pi S)_x,v)$ as follows:
	\begin{equation}\label{main decomposition}
		\begin{aligned}
			&(b(S-\Pi S)_x,v)\\
			=&\int_{[0,x_{\frac{N}{2}-1}]\times[0,1]}b(S-\Pi S)_x v+\int_{[x_{\frac{N}{2}-1},x_{\frac{N}{2}}]\times[0,1]}b(S-\Pi S)_x v+\int_{[x_{\frac{N}{2}},1]\times[0,1]}b(S-S^I)_x v.
		\end{aligned}
	\end{equation}
	
	Apply Green's formula to the first term on the right-hand side of \eqref{main decomposition}, we can get
	\begin{equation*}
		\int_{[0,x_{\frac{N}{2}-1}]\times[0,1]}b(S-\Pi S)_x v=-\int_{[0,x_{\frac{N}{2}-1}]\times[0,1]}(b_x v+bv_x)(S-\Pi S)+\int_{0}^{1}b(S-\Pi S)v(x_{\frac{N}{2}-1},y)\mathrm{d}y.
	\end{equation*}
	Lemma \ref{max-norm error } and Cauchy-Schwarz inequality yield
	\begin{align*}
		|-\int_{[0,x_{\frac{N}{2}-1}]\times[0,1]}(b_x v+bv_x)(S-\Pi S)|&\le C\sum_{i=0}^{\frac{N}{2}-2}\sum_{j=0}^{N-1}\Vert S-\Pi S\Vert_{K_{i,j}}\Vert v_x\Vert_{K_{i,j}}\\
		&\le C\varepsilon^{-\frac{1}{2}}N^{-2}(\sum_{i=0}^{\frac{N}{2}-2}\sum_{j=0}^{N-1}h_{x,i}h_{y,j})^{\frac{1}{2}}(\sum_{i=0}^{\frac{N}{2}-2}\sum_{j=0}^{N-1}\Vert v\Vert_{K_{i,j}})^{\frac{1}{2}}\\
		&\le CN^{-2}\ln^{\frac{1}{2}}{N}\Vert v\Vert_\varepsilon,
	\end{align*}
	and
	\begin{align*}
		|\int_{0}^{1}b(S-\Pi S)v(x_{\frac{N}{2}-1},y)\mathrm{d}y|&\le CN^{-2}\sum_{j=0}^{N-1}\int_{y_j}^{y_{j+1}}(\int_{0}^{x_{\frac{N}{2}-1}}v_x(x,y)\mathrm{d}x)\mathrm{d}y\\
		&\le CN^{-2}\sum_{j=0}^{N-1}\sum_{i=0}^{\frac{N}{2}-2}\Vert v_x\Vert_{L^1(K_{i,j})}\\
		&\le CN^{-2}\ln^{\frac{1}{2}}{N}\Vert v\Vert_\varepsilon.
	\end{align*}
	
	Before the estimate of the remained two terms of \eqref{main decomposition}, we first define $\Pi^N b$, the discrete $L_2-$projection of $b$, as
	\begin{equation*}
		b^{K_{i,j}}=\Pi^N b|_{K_{i,j}}=\frac{1}{h_{x,i}h_{y,j}}\int_{K_{i,j}}b.
	\end{equation*}
	See that $\Pi^N b$ is a piecewise constant vector function. It is a standard result that
	\begin{equation}\label{projection}
			\quad\Vert b-\Pi^N b\Vert_{\infty}\le CN^{-1}\vert b\vert_{1,\infty}.
	\end{equation}

	Furthermore, we will frequently make use of the following estimates in the subsequent analysis. For any $v\in V^N$ and $j=\frac{N}{4},\dots,\frac{3N}{4}-1$:
	\begin{align}
		&|v(x_{\frac{N}{2}-1},y_j)|\le Ch^{-1}\int_{y_j}^{y_{j+1}}(\int_{0}^{x_{\frac{N}{2}-1}}v_x(x,y)\mathrm{d}x)\mathrm{d}y\le Ch^{-\frac{1}{2}}\ln^{\frac{1}{2}}N\Vert v\Vert_{[0,x_{\frac{N}{2}-1}]\times[y_j,y_{j+1}]}\label{frequent 1},\\
		&|v(x_{\frac{N}{2}},y_j)|\le C\Vert v\Vert_{L^\infty(K_{\frac{N}{2},j})}\le C(h_{x,\frac{N}{2}}h_{y,j})^{-\frac{1}{2}}\Vert v\Vert_{K_{\frac{N}{2},j}}\le CN\Vert v\Vert_{K_{\frac{N}{2},j}}\label{frequent 2},
	\end{align}
\begin{equation}\label{inverse}
\begin{aligned}
	\Vert (S-\mathcal{P} S)_x\Vert_{L^\infty(K_{\frac{N}{2}-1,j})}&\le\Vert(S-S^I)_x\Vert_{L^\infty(K_{\frac{N}{2}-1,j})}+\Vert(S^I-\mathcal{P} S)_x\Vert_{L^\infty(K_{\frac{N}{2}-1,j})}\\
	&\le CN^{-1}+Ch_{x,\frac{N}{2}-1}^{-1}\Vert S^I-\mathcal{P} S\Vert_{L^\infty(K_{\frac{N}{2}-1,j})}\\
	&\le Ch_{x,\frac{N}{2}-1}^{-1}N^{-2}.
\end{aligned}
\end{equation}

	Now, we decompose the second term of \eqref{main decomposition} as 
	\begin{equation}\label{main 1}
		\begin{aligned}
			&\int_{[x_{\frac{N}{2}-1},x_{\frac{N}{2}}]\times[0,1]}b(S-\Pi S)_x v\\
			=&\int_{[x_{\frac{N}{2}-1},x_{\frac{N}{2}}]\times([0,y_{\frac{N}{4}}]\cup[y_{\frac{3N}{4}},1])}b(S-S^I)_xv+\int_{[x_{\frac{N}{2}-1},x_{\frac{N}{2}}]\times[y_{\frac{N}{4}},y_{\frac{3N}{4}}]}b(S-\mathcal{P}S)_x v\\
			=&\uppercase\expandafter{\romannumeral1}+\uppercase\expandafter{\romannumeral2}+\uppercase\expandafter{\romannumeral3}+\uppercase\expandafter{\romannumeral4},
		\end{aligned}
	\end{equation}
where
\begin{align}
	&\uppercase\expandafter{\romannumeral1}=:\int_{[x_{\frac{N}{2}-1},x_{\frac{N}{2}}]\times([0,y_{\frac{N}{4}}]\cup[y_{\frac{3N}{4}},1])}b(S-S^I)_xv,\\
	&\uppercase\expandafter{\romannumeral2}=:\sum_{j=\frac{N}{4}}^{\frac{3N}{4}-1}\int_{K_{\frac{N}{2}-1,j}}b^{K_{\frac{N}{2}-1,j}}(S-\mathcal{P}S)_x\left[v(x_{\frac{N}{2}-1},y_j)\theta_{\frac{N}{2}-1,j}+v(x_{\frac{N}{2}-1},y_{j+1})\theta_{\frac{N}{2}-1,j+1}\right],\\
	&\uppercase\expandafter{\romannumeral3}=:\sum_{j=\frac{N}{4}}^{\frac{3N}{4}-1}\int_{K_{\frac{N}{2}-1,j}}(b^{K_{\frac{N}{2}-1,j}}-b^{K_{\frac{N}{2},j}})(S-\mathcal{P}S)_x\left[v(x_{\frac{N}{2}},y_j)\theta_{\frac{N}{2},j}+v(x_{\frac{N}{2}},y_{j+1})\theta_{\frac{N}{2},j+1}\right],\\
	&\uppercase\expandafter{\romannumeral4}=:\sum_{j=\frac{N}{4}}^{\frac{3N}{4}-1}\int_{K_{\frac{N}{2}-1,j}}b^{K_{\frac{N}{2},j}}(S-\mathcal{P}S)_x\left[v(x_{\frac{N}{2}},y_j)\theta_{\frac{N}{2},j}+v(x_{\frac{N}{2}},y_{j+1})\theta_{\frac{N}{2},j+1}\right]\label{tough 1}.
\end{align}
Take use of the integral identity from \cite{Lin1Yan2:1996-Construction}: Let $K\in \mathcal{T}_{N}$ and suppose that $K$ is $K_{i,j}$, $(x_K, y_K)$ is the center of $K$, then
\begin{equation}
\int_{K}\frac{\partial}{\partial x}(S-S^I)v=\int_{K} R(S,v)+\frac{h_{x,i}^2}{12}\left[\int_{y_j}^{y_{j+1}}\frac{\partial^2 S}{\partial x^2}v(x_{i+1},y)\mathrm{d}y-\int_{y_j}^{y_{j+1}}\frac{\partial^2 S}{\partial x^2}v(x_{i},y)\mathrm{d}y\right],
\end{equation}
where
\begin{equation*}
\begin{aligned}
	R(S,v)=&\frac{1}{3}E(x)(x-x_K)\frac{\partial^3 S}{\partial x^3}\frac{\partial v}{\partial x}-\frac{h_{x,i}^2}{12}\frac{\partial^3 S}{\partial x^3}v+F(y)\frac{\partial^3 S}{\partial x\partial y^2}\\
	&\cdot\left[ v-(x-x_K)\frac{\partial v}{\partial x}-\frac{2}{3}(y-y_K)\frac{\partial v}{\partial y}+\frac{2}{3}(x-x_K)(y-y_K)\frac{\partial^2 v}{\partial x\partial y}\right],
\end{aligned}
\end{equation*}
and refer to \cite[(4.24)-(4.31)]{Zhang1:2003-Finite} for detailed information, the last term of \eqref{main decomposition} can be decomposed as
	\begin{equation}\label{main 2}
		\begin{aligned}
			&\int_{[x_{\frac{N}{2}},1]\times[0,1]}b(S-S^I)_x v\\
			=&\int_{[x_{\frac{N}{2}},1]\times[0,1]}(b-\Pi^N b)(S-S^I)_x v+\int_{[x_{\frac{N}{2}},1]\times[0,1]}\Pi^N b(S-S^I)_x v\\
			=&\uppercase\expandafter{\romannumeral5}+\uppercase\expandafter{\romannumeral6}+\uppercase\expandafter{\romannumeral7}+\uppercase\expandafter{\romannumeral8}+\uppercase\expandafter{\romannumeral9},
		\end{aligned}
	\end{equation}
	where
	\begin{align}
		&\uppercase\expandafter{\romannumeral5}=:\int_{[x_{\frac{N}{2}},1]\times[0,1]}(b-\Pi^N b)(S-S^I)_x v,\\
		&\uppercase\expandafter{\romannumeral6}=:\int_{[x_{\frac{N}{2}},1]\times[0,1]}\Pi^N b R(S,v),\\
		&\uppercase\expandafter{\romannumeral7}=:\sum_{i=\frac{N}{2}}^{N-2}\sum_{j=0}^{N-1}\frac{H^2}{12}(b^{K_{i,j}}-b^{K_{i+1,j}})\int_{y_j}^{y_{j+1}}\frac{\partial^2 S}{\partial x^2}v(x_{i+1},y)\mathrm{d}y,\\
		&\uppercase\expandafter{\romannumeral8}=:\sum_{j=0}^{\frac{N}{4}-1}-b^{K_{\frac{N}{2},j}}\frac{H^2}{12}\int_{y_j}^{y_{j+1}}\frac{\partial^2 S}{\partial x^2}v(x_{\frac{N}{2}},y)\mathrm{d}y+\sum_{j=\frac{3N}{4}}^{N-1}-b^{K_{\frac{N}{2},j}}\frac{H^2}{12}\int_{y_j}^{y_{j+1}}\frac{\partial^2 S}{\partial x^2}v(x_{\frac{N}{2}},y)\mathrm{d}y,\\
		&\uppercase\expandafter{\romannumeral9}=:\sum_{j=\frac{N}{4}}^{\frac{3N}{4}-1}-b^{K_{\frac{N}{2},j}}\frac{H^2}{12}\int_{y_j}^{y_{j+1}}\frac{\partial^2 S}{\partial x^2}v(x_{\frac{N}{2}},y)\mathrm{d}y\label{tough 2}.
	\end{align}
	
	Then we give the estimate of $\uppercase\expandafter{\romannumeral1}-\uppercase\expandafter{\romannumeral9}$. Recall Lemma \ref{interpolation error} and triangle inequality, one has
	\begin{equation}\label{term S1}
		\begin{aligned}
			|\uppercase\expandafter{\romannumeral1}|&\le |\int_{[x_{\frac{N}{2}-1},x_{\frac{N}{2}}]\times([0,y_{\frac{N}{4}-1}]\cup[y_{\frac{3N}{4}+1},1])}b(S-S^I)_xv|+|\int_{[x_{\frac{N}{2}-1},x_{\frac{N}{2}}]\times([y_{\frac{N}{4}-1},y_{\frac{N}{4}}]\cup[y_{\frac{3N}{4}},y_{\frac{3N}{4}+1}])}b(S-S^I)_xv|\\
			&\le CN^{-1}\left[(h_{x,\frac{N}{2}-1}\sqrt\varepsilon\ln{N})^{\frac{1}{2}}+(h_{x,\frac{N}{2}-1}h_{y,\frac{N}{4}-1})^{\frac{1}{2}}\right]\Vert v\Vert_\varepsilon\\
			&\le C\varepsilon^{\frac{1}{4}}N^{-\frac{3}{2}}\Vert v\Vert_\varepsilon.
		\end{aligned}
	\end{equation}
	Apply \eqref{frequent 1}, \eqref{inverse} and Cauchy-Schwarz inequality to obtain
	\begin{equation}\label{term S2}
		\begin{aligned}
			|\uppercase\expandafter{\romannumeral2}|&\le CN^{-2}\sum_{j=\frac{N}{4}}^{\frac{3N}{4}-1}h_{x,\frac{N}{2}-1}^{-1}h^{-\frac{1}{2}}\ln^{\frac{1}{2}}{N}\Vert v\Vert_{[0,x_{\frac{N}{2}-1}]\times[y_j,y_{j+1}]}h_{x,\frac{N}{2}-1}h\\
			&\le CN^{-2}\ln^{\frac{1}{2}}{N}(\sum_{j=\frac{N}{4}}^{\frac{3N}{4}-1}h)^{\frac{1}{2}}(\sum_{j=\frac{N}{4}}^{\frac{3N}{4}-1}\Vert v\Vert_{[0,x_{\frac{N}{2}-1}]\times[y_j,y_{j+1}]}^2)^{\frac{1}{2}}\\
			&\le CN^{-2}\ln^{\frac{1}{2}}{N}\Vert v\Vert_\varepsilon.
		\end{aligned}
	\end{equation}
	Estimate of \uppercase\expandafter{\romannumeral3} can be easily obtained by \eqref{projection}, \eqref{frequent 2}, \eqref{inverse} and Cauchy-Schwarz inequality.
	\begin{equation}\label{term S3}
		\begin{aligned}
			|\uppercase\expandafter{\romannumeral3}|&\le C\sum_{j=\frac{N}{4}}^{\frac{3N}{4}-1}N^{-3}h_{x,\frac{N}{2}-1}^{-1}N\Vert v\Vert_{K_{\frac{N}{2},j}}h_{x,\frac{N}{2}-1}h\\
			&\le CN^{-\frac{5}{2}}\Vert v\Vert_\varepsilon,
		\end{aligned}
	\end{equation}
	Identity \eqref{new interpolation} gives rise to
	\begin{equation}\label{term S4}
		\uppercase\expandafter{\romannumeral4}+\uppercase\expandafter{\romannumeral9}=\expandafter{\romannumeral1}+\expandafter{\romannumeral2}+\expandafter{\romannumeral3}+\expandafter{\romannumeral4},
	\end{equation}
	where
	\begin{align*}
		&\expandafter{\romannumeral1}=:-b^{K_{\frac{N}{2},\frac{N}{4}}}h_{x,\frac{N}{2}}^2\int_{y_{\frac{N}{4}}}^{y_{\frac{N}{4}+1}}\frac{\partial^2 S}{\partial x^2}v(x_{\frac{N}{2}},y_{\frac{N}{4}})\theta_{\frac{N}{2},\frac{N}{4}}(x_{\frac{N}{2}},y)\mathrm{d}y,\\
		&\expandafter{\romannumeral2}=:-b^{K_{\frac{N}{2},\frac{3N}{4}-1}}h_{x,\frac{N}{2}}^2\int_{y_{\frac{3N}{4}-1}}^{y_{\frac{3N}{4}}}\frac{\partial^2 S}{\partial x^2}v(x_{\frac{N}{2}},y_{\frac{3N}{4}})\theta_{\frac{N}{2},\frac{3N}{4}}(x_{\frac{N}{2}},y)\mathrm{d}y,\\
		&\expandafter{\romannumeral3}=:b^{K_{\frac{N}{2},\frac{N}{4}}}\int_{K_{\frac{N}{2}-1,\frac{N}{4}}}(S-\mathcal{P}S)_x v(x_{\frac{N}{2}},y_{\frac{N}{4}})\theta_{\frac{N}{2},\frac{N}{4}},\\
		&\expandafter{\romannumeral4}=:b^{K_{\frac{N}{2},\frac{3N}{4}-1}}\int_{K_{\frac{N}{2}-1,\frac{3N}{4}-1}}(S-\mathcal{P}S)_x v(x_{\frac{N}{2}},y_{\frac{3N}{4}})\theta_{\frac{N}{2},\frac{3N}{4}}.
	\end{align*}
	We have, from the standard approximation theory and \eqref{frequent 2}
	\begin{equation}\label{term S4.1}
		|\expandafter{\romannumeral1}+\expandafter{\romannumeral2}|\le CN^{-2}\Vert v\Vert_\varepsilon,
	\end{equation}
	and from \eqref{frequent 2} as well as \eqref{inverse}
	\begin{equation}\label{term S4.2}
		|\expandafter{\romannumeral3}+\expandafter{\romannumeral4}|\le CN^{-2}\Vert v\Vert_\varepsilon,
	\end{equation}
Recall \eqref{projection} and Lemma \ref{interpolation error}, one obtains
	\begin{equation}\label{term S5}
		|\uppercase\expandafter{\romannumeral5}|\le CN^{-2}\Vert v\Vert_\varepsilon.
	\end{equation}
 For \uppercase\expandafter{\romannumeral6}, inverse inequality and Hölder inequality yield 
	\begin{equation*}
	|\int_{K\subset[x_{\frac{N}{2}},1]\times[0,1]}\Pi^N b R(S,v)|\le CN^{-1}\Vert R(S,v)\Vert_K\le CN^{-3}\Vert v\Vert_K
	\end{equation*}
    Summing over $K\subset[x_{\frac{N}{2}},1]\times[0,1]$ and applying Cauchy-Schwarz inequality, we get
    \begin{equation}\label{term S6}
    |\uppercase\expandafter{\romannumeral6}|\le CN^{-3}\sum_{i=\frac{N}{2}}^{N-1}\sum_{j=0}^{N-1}\Vert v\Vert_{K_{i,j}} \le CN^{-2}\Vert v\Vert_\varepsilon.
    \end{equation}
	Trace inequality and Cauchy-Schwarz inequality yield
	\begin{equation}\label{term S7}
		\begin{aligned}
			|\uppercase\expandafter{\romannumeral7}|&\le CN^{-3}\sum_{i=\frac{N}{2}}^{N-2}\sum_{j=0}^{N-1}H^{-1}\Vert v\Vert_{L^1(K_{i,j})}\\
			&\le CN^{-3}\sum_{i=\frac{N}{2}}^{N-2}\sum_{j=0}^{N-1}H^{-\frac{1}{2}}h^{\frac{1}{2}}\Vert v\Vert_{K_{i,j}}\\
			&\le CN^{-2}\Vert v\Vert_\varepsilon.
		\end{aligned}
	\end{equation}
	For \uppercase\expandafter{\romannumeral8}, it is suffice to discuss the first term on the right-hand side, since the situation on the second term is the same as the first term. Triangle inequality, Trace inequality and Cauchy-Shwarz inequality give us
	\begin{equation}\label{term S8}
		\begin{aligned}
			&|\sum_{j=0}^{\frac{N}{4}-1}-b^{K_{\frac{N}{2},j}}\frac{H^2}{12}\int_{y_j}^{y_{j+1}}\frac{\partial^2 S}{\partial x^2}v(x_{\frac{N}{2}},y)\mathrm{d}y|\\
			\le&|\sum_{j=0}^{\frac{N}{4}-2}-b^{K_{\frac{N}{2},j}}\frac{H^2}{12}\int_{y_j}^{y_{j+1}}\frac{\partial^2 S}{\partial x^2}v(x_{\frac{N}{2}},y)\mathrm{d}y|+|-b^{K_{\frac{N}{2},\frac{N}{4}-1}}\frac{H^2}{12}\int_{y_{\frac{N}{4}-1}}^{y_{\frac{N}{4}}}\frac{\partial^2 S}{\partial x^2}v(x_{\frac{N}{2}},y)\mathrm{d}y|\\
			\le&CN^{-2}	\sum_{j=0}^{\frac{N}{4}-2}H^{-1}\Vert v\Vert_{L^1(K_{\frac{N}{2},j})}+CN^{-2}H^{-1}\Vert v\Vert_{L^1(K_{\frac{N}{2},\frac{N}{4}-1})}\\
			\le&CN^{-\frac{3}{2}}(\sum_{j=0}^{\frac{N}{4}-2}h_{y,j})^{\frac{1}{2}} (\sum_{j=0}^{\frac{N}{4}-2}\Vert v\Vert_{K_{i,j}}^2)^{\frac{1}{2}}+CN^{-2}\Vert v\Vert_{K_{\frac{N}{2},\frac{N}{4}-1}}\\
			\le&C\varepsilon^{\frac{1}{4}}N^{-\frac{3}{2}}\ln^{\frac{1}{2}}{N}\Vert v\Vert_\varepsilon.
		\end{aligned}
	\end{equation}

	Collect \eqref{term S1}-\eqref{term S8}, we are done.
\end{proof}	

Now we give the main theorem.
\begin{theorem}\label{main theorem}
Under the conditions that both Assumptions \ref{bound} and \ref{transposition} are true. Suppose that $u$ is the exact solution to \eqref{model problem}, $u^{I}$ is the standard Lagrange interpolation of $u$, and $u^{N}$ is the corresponding finite element solution. Then, one has
	\begin{equation*}
		\Vert u^{I}-u^{N}\Vert_{\varepsilon}\le C\varepsilon^{\frac{1}{4}}N^{-\frac{3}{2}}\ln^{\frac{1}{2}}{N}+CN^{-2}\ln^{\frac{1}{2}}N.
	\end{equation*}
\begin{proof}
Triangle inequality yields
\begin{equation}
	\Vert u^{I}-u^{N}\Vert_{\varepsilon}\le \Vert u^{I}-\varPi u\Vert_{\varepsilon}+\Vert \varPi u-u^{N}\Vert_{\varepsilon}.
\end{equation}

Collecting Lemma \ref{all term}, Lemma \ref{term E} and Lemma \ref{term S}, we have
\begin{equation}\label{T1}
\Vert\varPi u-u^N\Vert_{\varepsilon}\le C\varepsilon^{\frac{1}{4}}N^{-\frac{3}{2}}\ln^{\frac{1}{2}}{N}+CN^{-2}\ln^{\frac{1}{2}}N.
\end{equation}

Then, for $ \Vert u^{I}-\varPi u\Vert_{\varepsilon}$,
\begin{equation*}
\Vert u^{I}-\varPi u\Vert_{\varepsilon}\le \Vert E_1^I-\pi E_1\Vert_\varepsilon+\Vert S^I-\Pi S\Vert_\varepsilon.
\end{equation*}
Lemma \ref{error of E} yields
\begin{equation}\label{T2}
\Vert E_1^I-\pi E_1\Vert_\varepsilon=\Vert\mathcal{Q}E_1-\mathcal{B}E_1\Vert_\varepsilon\le C(1+\varepsilon^{\frac{1}{4}}N^{\frac{1}{2}})N^{-\sigma}.
\end{equation}
\eqref{dot error} and inverse inequality yield
\begin{equation}\label{T3}
\begin{aligned}
\Vert S^I-\Pi S\Vert_\varepsilon&=\Vert\sum_{j=\frac{N}{4}+1}^{\frac{3N}{4}-1}(S^I-\mathcal{P}S)(x_{\frac{N}{2}-1},y_j)\theta_{\frac{N}{2}-1,j}\Vert_\varepsilon\\
&\le CN^{-2}\sum_{j=\frac{N}{4}+1}^{\frac{3N}{4}-1}\Vert \theta_{\frac{N}{2}-1,j}\Vert_\varepsilon\\
&\le CN^{-2}\sum_{j=\frac{N}{4}+1}^{\frac{3N}{4}-1}(h_{x,\frac{N}{2}-1}^{-1}h_{x,\frac{N}{2}-1}h+h^{-1}h_{x,\frac{N}{2}-1}h+h_{x,\frac{N}{2}-1}h)\\
&\le CN^{-2}.
\end{aligned}
\end{equation}

Collecting \eqref{T1}, \eqref{T2} and \eqref{T3}, we are done.
\end{proof}
\end{theorem}

\section{Numerical results}\label{sec. 5}
Here we will do some numerical experiments to support our theoretical results. Calculations are performed by Intel Visual Fortran 11, and we can refer to \cite{Ben1Mic2:2005-Numerical} for the discrete problems.

Consider the following test problem of \eqref{model problem}:
\begin{equation*}
	\begin{aligned}
		-\varepsilon\Delta u-(3-x-y)u_{x}+2u=&f(x,y)\quad&&\text{in $\Omega=(0,1)^{2}$},\\
		u=&0\quad &&\text{on $\partial\Omega$}.
	\end{aligned}
\end{equation*}
We choose $f(x,y)$ such that 
\begin{equation*}
	u(x,y)=(\cos\frac{\pi x}{2}-\frac{e^{-\frac{x}{\varepsilon}}-e^{-\frac{1}{\varepsilon}}}{1-e^{-\frac{1}{\varepsilon}}})(\frac{(1-e^{-\frac{y}{\sqrt\varepsilon}})(1-e^{-\frac{1-y}{\sqrt\varepsilon}})}{1-e^{-\frac{y}{\sqrt\varepsilon}}}).
\end{equation*}
is the exact solution to \eqref{model problem}.

In our investigation, we use the bilinear FEM and assume that $\varepsilon\le N^{-1}$. Numerical results can be found in Table \ref{table:1}, which lists errors and convergence order under the energy norm $\Vert u^{I}-u^{N}\Vert_{\varepsilon}$ in the case of $\varepsilon=10^{-2},10^{-3},\dots,10^{-8}$ and $N=8,16,32,64,128,,256$.

Table \ref{table:1} indicates that $\Vert u^I-u^N\Vert_\varepsilon$ converges at a rate of almost $\mathcal{O}(N^{-2})$, verifying Theorem \ref{main theorem}.
\begin{table}[http]
	\caption{Errors of $\Vert u^{I}-u^N\Vert_{\varepsilon}$ and convergence order}
	\footnotesize
	\begin{tabular*}{\textwidth}{@{}@{\extracolsep{\fill}} c cccccc @{}}
		\cline{1-7}{}
		\multirow{2}{*}{ $\varepsilon$ }&\multicolumn{6}{c}{$N$ }\\ 
		\cline{2-7}                 &8        &16         &32        &64       &128        &256           \\
		\cline{1-7}
		\multirow{2}{*}{ $10^{-2}$ }      &0.132E-01  &0.167-02    &0.209E-03  &0.264E-04 &0.334E-05    &0.426E-06  \\
		&2.99     &2.99      &2.99     &2.98    &2.97      &---       \\
		\cline{2-7}
		\multirow{2}{*}{ $10^{-3}$ }&0.223E-01  &0.336E-02    &0.386E-03   &0.439E-04 &0.525E-05   &0.647E-06  \\
		&2.73   &3.12       &3.14      &3.06    &3.02   &---       \\
		\cline{2-7}
		\multirow{2}{*}{ $10^{-4}$ }&0.281E-01  &0.295E-02    &0.353E-03   &0.498E-04 &0.801E-05   &0.117E-05 \\
		&3.25    &3.07       &2.82      &2.64    &2.77   &---       \\
		\cline{2-7}
		\multirow{2}{*}{ $10^{-5}$ }&0.235E-01  &0.266E-02    &0.339E-03   &0.508E-04 &0.941E-05   &0.199E-05  \\
		&3.14    &2.98       &2.74    &2.43    &2.24   &---      \\    
		\cline{2-7}
		\multirow{2}{*}{ $10^{-6}$ }&0.208E-01  &0.249E-02    &0.329E-03   &0.504E-04 &0.952E-05   &0.212E-05  \\
		&3.06    &2.92     &2.71     &2.40     &2.17     &---      \\ 
		\cline{2-7}
		\multirow{2}{*}{ $10^{-7}$}&0.195E-01  &0.241E-02    &0.324E-03   &0.501E-04 &0.952E-05   &0.213E-05 \\
		&3.02    &2.90      &2.69     &2.40     &2.16     &---    \\
		\cline{2-7}
		\multirow{2}{*}{ $10^{-8}$}&0.191E-01  &0.239E-02    &0.323E-03   &0.501E-04 &0.953E-05   &0.213E-05 \\
			&3.00    &2.89      &2.69     &2.39     &2.16     &---      \\ 
		\cline{1-7}{}
	\end{tabular*}
	\label{table:1}
\end{table}
\section{Conflict of interest statement}
We declare that we have no conflict of interest.
\section{Bibliography}
\bibliographystyle{plain}

\end{document}